\documentclass[reqno]{amsart}

\usepackage{enumerate}
\usepackage{hyperref}

\numberwithin{equation}{section}

\newtheorem{thm}{Theorem}[section]
\newtheorem{lem}[thm]{Lemma}
\newtheorem{prop}[thm]{Proposition}
\newtheorem{cor}[thm]{Corollary}

\theoremstyle{definition}
\newtheorem{rem}[thm]{Remark}
\newtheorem{claim}[thm]{Claim}

\newcommand\R{\mathbb {R}}
\newcommand\N{\mathbb {N}}

\newcommand\DI{\varphi }
\newcommand\Pwr{{\alpha +1}}
\newcommand\Pmu{{\alpha }}
\newcommand\Ppu{{\alpha +2}}
\newcommand\Swr{{\boldsymbol{s} }}
\newcommand\Awr{{\boldsymbol{a}}}
\newcommand\Rwr{{\boldsymbol{r}}}
\newcommand\Bwr{{\boldsymbol{b}}}
\newcommand\Qwr{{\boldsymbol{q}}}
\newcommand\Gwr{{\boldsymbol{\gamma }}}
\newcommand\Lwr{{\boldsymbol{\lambda }}}
\newcommand\Siwr{{\boldsymbol{\sigma}}}
\newcommand\Dim{{N}}
\newcommand\Cgn{{C _{ \mathrm {GN} }}}
\newcommand\Dsd{{\delta _{ \mathrm {sd} }}}
\newcommand\Csd{{C _{ \mathrm {sd} }}}
\newcommand\Fcr{{\varphi  _{ \mathrm {crit} }}}
\newcommand\Ucr{{u _{ \mathrm {crit} }}}
\newcommand\Step[1]{\noindent {\it Step~#1.}\quad}
\newcommand\Srn{{\mathcal S}(\R^\Dim )}

\newcommand\Es[1]{e^{i #1 \Delta }}
\newcommand\Esm[1]{e^{-i #1 \Delta }}

\newcommand\Inv{{\mathcal K}}
\newcommand\Nonad{{\mathcal L}}

\newcommand\Tma{T_{\mathrm{max}}}
\newcommand\im{{\mathrm{Im}}\,}

\newcommand\goto{\mathop{\longrightarrow}}
\newcommand\union{\mathop{\cup}}
\newcommand\weakcv{\mathop{\rightharpoonup}}
\newcommand\Comp{{\mathrm{c}}}

\newcommand\Sol{{\boldsymbol{\mathcal S}}}

\begin{document}

\title{Scattering for the focusing energy-subcritical NLS}

\author[Daoyuan Fang]{Daoyuan Fang$^{1}$}

\author[Jian Xie]{Jian Xie$^{1}$}

\author[Thierry Cazenave]{Thierry Cazenave$^{2}$}

\address{$^{1,2}$
Department of Mathematics, Zhejiang University, Hangzhou, Zhejiang 310027, People's Republic of China}

\curraddr[Thierry Cazenave]{{\sc Universit\'e Pierre et Marie Curie \& CNRS, Laboratoire Jacques-Louis Lions,
B.C.~187, 4 place Jussieu, 75252 Paris Cedex 05, France}}

\email[Daoyuan Fang]{\href{mailto:dyf@zju.edu.cn}{dyf@zju.edu.cn}}
\email[Jian Xie]{\href{mailto:sword711@gmail.com}{sword711@gmail.com}}
\email[Thierry Cazenave]{\href{mailto:thierry.cazenave@upmc.fr}{thierry.cazenave@upmc.fr}}

\thanks{$^1$Research supported by NSFC 10871175, 10931007, and Zhejiang NSFC Z6100217}
\thanks{$^2$Research supported in part by Zhejiang University's Pao Yu-Kong
International Fund}



\begin{abstract}
For the 3D focusing cubic nonlinear Schr\"o\-din\-ger  equation,
Scattering of $H^1$ solutions  inside the (scale invariant) potential well was established by Holmer and Roudenko~\cite{HR2} (radial case) and
Duyckaerts, Holmer and Roudenko~\cite{DHR} (general case).  In this paper, we
extend this result to arbitrary space dimensions and focusing, mass-supercritical and energy-subcritical  power nonlinearities, by adapting the method of~\cite{DHR}.
\end{abstract}

\maketitle
\section{ Introduction }
In this paper, we study the  Cauchy problem for the
focusing, mass supercritical and energy subcritical nonlinear Schr\"o\-din\-ger
equation in $\R^\Dim $, $\Dim \ge 1$
\begin{equation} \label{SCH} \tag{NLS}
\begin{cases}
iu_t + \Delta  u + |u|^{\Pmu }u = 0,\\
u(0) = \DI ,
\end{cases}
\end{equation}
where
\begin{equation} \label{fPwr}
\frac {4} {\Dim} < \Pmu <
\begin{cases}
\displaystyle \frac {4} {\Dim -2}&  \text{if } \Dim \ge 3, \\ \infty &  \text{if }  \Dim =1,2.
\end{cases}
\end{equation}
As is well-known, the Cauchy problem~\eqref{SCH} is locally well-posed  in $H^1(\R^\Dim )$.
(See e.g.~\cite{GinibreVu,Katod}.) More precisely, given $\DI  \in H^1(\R^\Dim )$, there exists $T>0$ and a unique solution $u\in C([0,T], H^1(\R^\Dim ))$ of~\eqref{SCH}.  This solution can be extended to a maximal existence interval $[0, \Tma)$, and either $\Tma=\infty $ (global solution) or else $\Tma <\infty $ and $ \|u(t)\| _{ H^1 }\to \infty $ as $t\uparrow \Tma$ (finite time blowup).
Moreover, the mass $M(u(t))$, energy $E(u(t))$ and momentum $P(u(t))$ are independent of $t$, where
\begin{align}
M(w) &= \int  _{ \R^\Dim } |w(x)|^2 dx, \label{fMass}  \\
E(w) &= \frac {1} {2} \int  _{ \R^\Dim } |\nabla w(x)|^2 dx- \frac {1} {\Ppu } \int  _{ \R^\Dim } | w(x)|^{
\Ppu } dx, \label{fEner} \\
P(w) &=\im  \int  _{ \R^\Dim }  \overline{w} (x)\nabla w(x)\, dx, \label{fMom}
\end{align}
for all $w\in H^1(\R^\Dim )$.
Throughout this paper, we will only consider solutions in the above sense ($H^1$ solutions).

In the defocusing case, i.e. when the nonlinearity $ |u|^{p-1}u$ in~\eqref{SCH} is replaced by  $- |u|^{p-1}u$, it is well-known~\cite{GinibreVd, Nakanishi} that all solutions scatter as $t\to \infty $,
i.e. there exists a scattering state $u^+\in H^1 (\R^\Dim ) $ such that
\begin{equation} \label{fDScat}
 \|u(t) - \Es{t}  u^+\| _{ H^1 } \goto  _{ t\to \infty  } 0.
\end{equation}
In the focusing case, the situation is much richer: solutions with small initial values scatter~\cite[Theorem~17]{Straussd}, but some solutions blow up in finite time~\cite{Glassey},
 and some solutions are global and do not scatter. A typical example of solutions of the latter type are standing waves, i.e. solutions of the form $u(t,x)= e^{i\omega t} \varphi (x)$. When $\omega >0$ such solutions exist~\cite{Straussu} and correspond to $\varphi (x)= \omega ^{\frac {1} {\alpha }} \psi ( \omega ^{\frac {1} {2}} x)$ where $\psi \in H^1(\R^\Dim )$ is a nontrivial solution
 of the elliptic problem
 \begin{equation} \label{fGrSt}
-\Delta w+w=  |w|^{\Pmu }w.
\end{equation}
Standing waves of particular interest are given by the ground states, i.e. the solutions of~\eqref{fGrSt}  which minimize the energy~\eqref{fEner} among all nontrivial $H^1$ solutions of~\eqref{fGrSt}. Ground states exist and are all of the form $e^{i\theta } Q(\cdot -y)$ where $\theta \in \R$, $y\in \R^\Dim $ and $Q$ is the unique, positive and radially decreasing
solution of~\eqref{fGrSt}. Moreover, $E(Q)>0$. (See~\cite{BerestyckiL, GidasNN, Kwong}.)

No necessary and sufficient condition on the initial value $\DI $ is known for the solution $u$ of~\eqref{SCH} to be global or blow up in finite time.
 Holmer and Roudenko~\cite[Theorem~2.1]{HR1} have obtained the following  sufficient condition.
Let
\begin{equation} \label{fSig}
\Siwr = \frac {4-(\Dim-2)\Pmu } {\Dim \Pmu -4} >0,
\end{equation}
and set
\begin{multline} \label{fDfInv}
\Inv = \{ u\in H^1 (\R^\Dim );\,
E(u) M(u)^ \Siwr < E(Q) M(Q)^\Siwr \\ \text{ and } \|\nabla u\| _{ L^2 } \|u\| _{ L^2 }^ \Siwr <
 \|\nabla Q\| _{ L^2 } \|Q\| _{ L^2 }^ \Siwr  \},
\end{multline}
where $Q$ is the ground state defined above.
It follows that $\Inv $ is invariant by the flow of~\eqref{SCH} and  that
every initial value $\DI \in \Inv$ produces a global solution of~\eqref{SCH} which is bounded in $H^1 (\R^\Dim )$ as $t\to \infty $. This condition is optimal in the sense that if
$E(\DI ) M(\DI )^ \Siwr < E(Q) M(Q)^ \Siwr $ and $ \|\nabla \DI \| _{ L^2 } \|\DI \| _{ L^2 }^ \Siwr >
 \|\nabla Q\| _{ L^2 } \|Q\| _{ L^2 }^ \Siwr $, then the resulting solution of~\eqref{SCH} blows up in finite time provided $ |\cdot |\DI \in L^2(\R^\Dim )$. (This observation applies for example to $\DI  (\cdot )= \lambda ^{\frac {2} {\Pmu }} Q(\lambda \cdot )$ with $\lambda >1$.)
 Note that $\Inv $ is the scale-invariant version of the potential well constructed in~\cite{BerestyckiC},
 with respect to the natural scaling of the equation
 \begin{equation}\label{scaling}
u_\lambda = \lambda^\frac{2 }{\Pmu } u(\lambda^2 t, \lambda x).
\end{equation}
Since $\Inv $ contains a neighborhood of $0$, and solutions with small initial values scatter, it is natural to ask whether or not solutions with initial values in $\Inv $ also scatter.
A positive answer to a similar question in the energy-critical case $\Pmu = \frac {4} {\Dim -2}$ has been given  by Kenig and Merle~\cite{KenigM} for radial solutions and $\Dim= 3,4,5$. (That result was extended by Killip and Visan~\cite{KillipV} to general solutions for $\Dim \ge 5$.)
For the cubic 3D equation ($\Dim=\Pwr = 3$), a positive answer is given in~\cite{HR2} (radial case) then in~\cite{DHR} (general case). In this paper, we extend the result of~\cite{DHR} to arbitrary $\Dim \ge 1$ and $\Pmu $ satisfying~\eqref{fPwr}. Our main result is the following.

\begin{thm}\label{main}
Let $\Dim \ge 1$, assume~\eqref{fPwr} and let $\Inv $ be defined by~\eqref{fDfInv}.
Given any $\DI \in \Inv $, it follows that the corresponding solution $u$ of~\eqref{SCH} is global and scatters, i.e. there exists a scattering state $u^+\in H^1 (\R^\Dim ) $ such that~\eqref{fDScat} holds.
\end{thm}

The conclusion of Theorem~\ref{main} is that for all initial value in $\Inv $, the corresponding solution of~\eqref{SCH} is global and scatters as $t\to \infty $. We note that for every $\DI \in \Inv $, the solution is also global for negative times and scatters as $t\to -\infty $. This follows from the fact that $\Inv $ is invariant by complex conjugation and that if $u (t)$ is a solution of~\eqref{SCH} for $t\ge 0$, then $ \overline{u} (-t) $ is a solution for $t\le 0$ (with the initial value $ \overline{\DI } $).

The rest of the paper is organized as follows. In Section~\ref{sProof}, we give the sketch of the proof of Theorem~\ref{main}, assuming all the technical results. In Section~\ref{sEnergy} we recall some energy inequalities related to the set $\Inv$ and in Section~\ref{sLocal} we recall some results on the Cauchy problem~\eqref{SCH}.
Section~\ref{sProfile} is devoted to a profile decomposition theorem and
Section~\ref{sCritSol} to the construction of a particular ``critical" solution.
Finally, we prove in Section~\ref{sRigidity} a rigidity theorem.
The appendix (Section~\ref{sGron}) contains a Gronwall-type lemma, which we use in Section~\ref{sLocal}.

\medskip

 {\bf Notation:}\quad Throughout this paper, we use the following notation. We denote by $p'$ the conjugate of the exponent $p\in [1,\infty ]$ defined by $\frac {1} {p} +\frac {1} {p'}=1$.
We will use the (complex valued) Lebesgue space $L^p (\R^\Dim ) $ (with norm $ \|\cdot \| _{ L^p }$) and Sobolev spaces $H^s (\R^\Dim ) $, $\dot H^s (\R^\Dim ) $ (with respective norms $ \|\cdot \| _{ H^s }$ and $ \|\cdot \| _{ \dot H^s }$). Given any interval $I\subset \R$, the norm of the mixed space $ L^a(I,L^b(\R^\Dim )) $ is denoted by  $ \|\cdot \| _{ L^a(I,L^b) }$.
We denote by  $(\Es{t}) _{ t\in \R }$ the Schr\"o\-din\-ger group, which is isometric on $H^s$ and $\dot H^s$ for every $s\ge 0$.
We will use freely the well-known properties of $(\Es{t}) _{ t\in \R }$. (See e.g. Chapter~2 of~\cite{Css} for an account of these properties.)
We define  the following set of indices, depending only on $\Dim$ and $\Pmu $,
\begin{gather*}
\Awr =\frac {2 \Pmu (\Ppu )} {4 - (\Dim  - 2) \Pmu  }, \quad \Bwr  = \frac {2 \Pmu (\Ppu )} {\Dim \Pmu ^2 + (\Dim  - 2)\Pmu  - 4},
\quad \Gwr =\frac {2(\Dim +2)} {\Dim},\\
\Qwr = \frac {4 (\Ppu )} {\Dim \Pmu }, \quad
\Rwr  =\Ppu ,\quad
\Swr  = \frac{\Dim }{2} - \frac{2}{\Pmu } \in  \Bigl( 0, \min \Bigl\{1, \frac {N} {2} \Bigr\} \Bigr),
\end{gather*}
and we note for further use that $(\Pwr )\Rwr '= \Rwr,$ and $ (\Pwr ) \Bwr '= \Awr$.

\section{Sketch of the proof of Theorem~\ref{main}} \label{sProof}
In this short section, we give the proof of Theorem~\ref{main}, assuming all preliminary results.
Our proof is adapted from the proof of~\cite{DHR} concerning the cubic 3D equation, but follows more closely the proof of~\cite{KenigM} of the energy-critical case.
Note, however, that the argument in the present situation is considerably simpler than the argument in~\cite{KenigM} for at least two reasons.
The equation~\eqref{SCH} is energy subcritical;
and  we consider solutions not only with finite energy, but also with finite mass.
This restriction is essential so that the set $\Inv$ makes sense.
The main ingredients are a profile decomposition (Theorem~\ref{ProfileE}) and a Liouville-type  (Theorem~\ref{Rigidity}) result, and the construction of a ``critical" solution (Proposition~\ref{critical}). These results, as well as all other technical results, are stated and proved in the following sections.

We define
\begin{equation} \label{fProofd}
\Nonad= \{ \DI \in \Inv ;\, u\in L^\Awr ((0,\infty ), L^\Rwr (\R^\Dim )) \},
\end{equation}
where $u$ is  the solution of~\eqref{SCH} with initial value $\DI $.
It is well-known (see Proposition~\ref{eBasic}) that if $\DI \in \Nonad$, then  $u$ scatters.
Therefore, we need only show that $\Inv = \Nonad$.

Given $0<\omega < 1$, we set
\begin{equation} \label{fProoft}
\Inv  _{ \omega  }= \{ u\in \Inv;\,  E(u) M(u)^ \Siwr \le  \omega  E(Q) M(Q)^\Siwr\}.
\end{equation}
In particular, $\Inv =  \union _{ 0<\omega <1 } \Inv _{ \omega  }$.
Since by~\eqref{fCdeltad}, $\|\nabla \DI \| _{ L^2 }  \|\DI \| _{ L^2 }^\Siwr \le \omega ^{\frac {1} {2}} \|\nabla Q\| _{ L^2 }  \|Q\| _{ L^2 }^\Siwr $ for all $\DI \in \Inv  _{ \omega  }$, we deduce from the elementary interpolation inequality $ \| \DI \| _{ \dot H^\Swr }^{1+ \Siwr} \le  \|\nabla \DI \| _{ L^2 }  \|\DI \| _{ L^2 }^\Siwr$ that $ \| \DI \| _{ \dot H^\Swr }^{1+ \Siwr} \le \omega ^{\frac {1} {2}} \|\nabla Q\| _{ L^2 }  \|Q\| _{ L^2 }^\Siwr $ for all $\DI \in \Inv  _{ \omega  }$.
Applying Proposition ~\ref{small}, we conclude that $\Inv  _{ \omega  }\subset \Nonad$ for all sufficiently small $\omega >0$.
Therefore, we may define the number
\begin{equation} \label{fProofq}
\omega _0= \sup \{ \omega \in (0,1);\, \Inv  _{ \omega  }\subset \Nonad \} \in (0,1],
\end{equation}
and we need only show that $\omega _0=1$.

We argue by contradiction and assume $0<\omega _0<1$.
The first crucial step of the proof (Proposition~\ref{critical}), which is based on the profile decomposition theorem,  shows that there exists an initial value $\Fcr \in \Inv _{ \omega _0 }$ which is not in $\Nonad$.
Moreover, the corresponding solution $\Ucr $ of~\eqref{SCH} has the following compactness property: there exists a function $x\in C([0,\infty ), \R^\Dim )$ such that $\union  _{ t\ge 0 } \{ \Ucr (t, \cdot -x(t)) \}$ is relatively compact in $H^1 (\R^\Dim ) $.
It follows (see Lemma~\ref{eRCpu}) that
\begin{equation*}
\sup  _{ t\ge 0 } \int  _{ \{  |x+x(t)|>R \} }\{  |\nabla \Ucr (t,x)|^2 +  | \Ucr (t,x)|^{\Ppu } + | \Ucr (t,x)|^2 \}
\goto  _{ R\to \infty  }0.
\end{equation*}
The second crucial step (Theorem~\ref{Rigidity}) shows that necessarily $\Fcr =0$.
Since $\Fcr \not \in  \Nonad$, we obtain a contradiction. Thus $\omega _0= 1$ and the proof is complete.

\section{Energy inequalities} \label{sEnergy}
In this section, we collect certain energy inequalities that are related to the set $\Inv$.
We recall that the best constant in the Gagliardo-Nirenberg inequality
\begin{equation}\label{GN}
\| f \|_{L^{\Ppu }}^{\Ppu } \le \Cgn \| f \|_{L^2}^{\frac{ 4- (\Dim  - 2) \Pmu  }{2}}\| \nabla  f \|_{L^2}^\frac{\Dim \Pmu }{2},
\end{equation}
is given by
\begin{equation}\label{CGN}
\Cgn =\frac{\| Q \|_{L^{\Ppu }}^{\Ppu }}{\| Q \|_{L^2}^{\frac{ 4- (\Dim  - 2) \Pmu  }{2}}\| \nabla  Q \|_{L^2}^\frac{\Dim \Pmu }{2}},
\end{equation}
where $Q$ is the ground state of~\eqref{fGrSt}. (See Corollary~2.1
in~\cite{Wei}.) We also recall that by Pohozaev's identity,
\begin{equation}\label{Q}
\begin{split}
\| Q \|_{L^2}^2 & = \frac {4- (\Dim -2) \Pmu } {\Dim \Pmu } \|\nabla Q \|_{L^2}^2
=\frac {4- (\Dim -2) \Pmu } {2(\Ppu )} \| Q \|_{L^{\Ppu }}^{\Ppu }
\\ & = \frac {8- 2(\Dim -2) \Pmu } {\Dim \Pmu -4}  E(Q),
\end{split}
\end{equation}
(see e.g. Corollary~8.1.3 in~\cite{Css}) and we note that by~\eqref{CGN} and~\eqref{Q}
\begin{equation} \label{fCGNb}
\Cgn =
 \frac {2(\Ppu )}  { \Dim \Pmu }
  [\| \nabla  Q \|_{L^2} \| Q \|_{L^2}^\Siwr]^{- \frac {\Dim \Pmu  -4} {2}}.
\end{equation}

\begin{lem}\label{Cdelta}
If $u\in H^1 (\R^\Dim ) $ and $ \|\nabla u\| _{ L^2 }  \|u\| _{ L^2 }^\Siwr \le  \|\nabla Q\| _{ L^2 }  \|Q\| _{ L^2 }^\Siwr$, then
\begin{gather}
E(u)\ge  \frac {\Dim \Pmu  -4} {2 \Dim \Pmu } \|\nabla u\| _{ L^2 }^2,  \label{fCdeltau} \\
\|\nabla u\| _{ L^2 }  \|u\| _{ L^2 }^\Siwr \le   \Bigl( \frac {E(u)M(u)^\Siwr } {E(Q)M(Q)^\Siwr } \Bigr)^{ \frac {1} {2} }  \|\nabla Q\| _{ L^2 }  \|Q\| _{ L^2 }^\Siwr, \label{fCdeltad} \\
8\| \nabla  u \|_{L^2}^2 - \frac{4 \Dim \Pmu }{\Ppu }\| u \|_{L^{\Ppu }}^{\Ppu }
\ge 8  \Bigl[ 1 -   \Bigl( \frac {E(u)M(u)^\Siwr } {E(Q)M(Q)^\Siwr } \Bigr)^{ \frac {\Dim \Pmu  -4} {4} }  \Bigr] \|\nabla u\| _{ L^2 }^2. \label{fCdeltat}
\end{gather}
In particular, the above estimates hold for all $u\in \Inv$,
where $\Inv $ is defined by~\eqref{fDfInv}.
\end{lem}

\begin{proof}
It follows from Gagliardo-Nirenberg's inequality~\eqref{GN} that
\begin{equation} \label{fCpbd}
\begin{split}
E(u)M(u)^\Siwr & = \frac {1} {2} [ \|\nabla u\| _{ L^2 }  \|u\| _{ L^2 }^\Siwr]^2
- \frac {1} {\Ppu }  \|u\| _{ L^{\Ppu  }} ^{\Ppu }   \|u\| _{ L^2 }^{2\Siwr}\\
& \ge f( \|\nabla u\| _{ L^2 }  \|u\| _{ L^2 }^\Siwr ),
\end{split}
\end{equation}
where
\begin{equation}  \label{fCpbt}
f(x)= \frac {1} {2}x^2- \frac {\Cgn} {\Ppu } x^{\frac {\Dim \Pmu } {2}},
\end{equation}
 for $x\ge 0$.
 Since equality holds in~\eqref{GN} for $f=Q$, we also have
\begin{equation}  \label{fCpbtb}
E(Q)M(Q)^\Siwr = f( \|\nabla Q\| _{ L^2 }  \|Q\| _{ L^2 }^\Siwr ).
\end{equation}
Note that $f$ is increasing on $(0,x_1)$ and decreasing on $(x_1,\infty )$, where
\begin{equation}  \label{fCpbq}
x_1= \Bigl(\frac {\Cgn} {\Ppu } \frac {\Dim \Pmu } {2} \Bigr)^{ - \frac {2} { \Dim \Pmu
-4}}=  \| \nabla  Q \|_{L^2} \| Q \|_{L^2}^\Siwr .
\end{equation}
(We used~\eqref{fCGNb} in the last identity.)
An elementary calculation shows that
\begin{equation}  \label{fCpbs}
f(x) \ge \frac {\Dim \Pmu  -4} {2 \Dim \Pmu } x^2,
\end{equation}
for $0\le x\le x_1$, with equality for $x=x_1$, i.e.
\begin{equation}  \label{fCpbp}
 f( \|\nabla Q\| _{ L^2 }  \|Q\| _{ L^2 }^\Siwr )=  \frac {\Dim \Pmu  -4} {2 \Dim \Pmu } ( \|\nabla Q\| _{ L^2 }  \|Q\| _{ L^2 }^\Siwr)^2.
\end{equation}
Therefore, we see that~\eqref{fCdeltau} follows from~\eqref{fCpbq}, \eqref{fCpbd} and~\eqref{fCpbs} and that~\eqref{fCdeltad} follows from~\eqref{fCpbd}, \eqref{fCpbs}, \eqref{fCpbtb} and~\eqref{fCpbp}.
Set now
\begin{equation*}
R(u)= 8\| \nabla  u \|_{L^2}^2 - \frac{4 \Dim \Pmu }{\Ppu }\| u \|_{L^{\Ppu }}^{\Ppu }.
\end{equation*}
If $ x= \|\nabla u\| _{ L^2 }  \|u\| _{ L^2 }^\Siwr$,
it follows from~\eqref{GN} and~\eqref{fCpbq}  that
\begin{equation*}
\begin{split}
\frac {1} {8} M(u)^\Siwr R(u)& \ge
x^2 - \frac{ \Dim \Pmu }{2 (\Ppu )} \Cgn x^{ \frac {\Dim \Pmu } {2}}
\\ & =x^2 - x_1 ^{- \frac {\Dim \Pmu -4} {2}} x ^{ \frac {\Dim \Pmu} {2}} = x^2 [1 - (x/x_1) ^{ \frac {\Dim \Pmu -4} {2}}  ].
\end{split}
\end{equation*}
Applying~\eqref{fCdeltad}, we deduce that
\begin{equation*}
\frac {1} {8} M(u)^\Siwr R(u) \ge  x^2  \Bigl[ 1 -   \Bigl( \frac {E(u)M(u)^\Siwr } {E(Q)M(Q)^\Siwr } \Bigr)^{ \frac {\Dim \Pmu -4} {4} }  \Bigr],
\end{equation*}
from which~\eqref{fCdeltat} follows.
\end{proof}

\begin{rem} \label{eRemd}
Let $0<\omega <1$ and let $\Inv  _{ \omega  }$ be defined by~\eqref{fProoft}.
It follows  from~\eqref{fCdeltad} that $\Inv  _{ \omega  }$ is a closed subset of $H^1 (\R^\Dim ) $.
\end{rem}

\section{The Cauchy problem} \label{sLocal}
We collect in this section some results concerning the Cauchy problem~\eqref{SCH}.
We will use in particular the Strichartz estimates
\begin{gather}
\|\Es{\cdot } \DI  \| _{ L^\Qwr (\R, L^\Rwr) } + \|\Es{\cdot } \DI  \| _{ L^\Gwr (\R, L^\Gwr) }
 \le C  \|\DI  \| _{ L^2 }, \label{fStru} \\
\| \Es{\cdot }  \DI  \|_{ L^\Awr (\R, L^\Rwr) }  \le C \| \DI  \|_{\dot{H}^\Swr }. \label{fStrt}
\end{gather}
\eqref{fStru} is a standard Strichartz estimate.  (See~\cite{Strichartz, Yajima}.)
\eqref{fStrt} follows by applying $(-\Delta )^{\frac {\Swr} {2}}$ to the standard Strichartz estimate
and using Sobolev's inequality.
We need the following Strichartz-type estimate for non admissible pairs
(see~\cite[Lemma~2.1]{RapDec})
\begin{equation}\label{KATO1}
\Bigl\| \int^t_0 \Es{(\cdot -t')} f(t') dt'  \Bigr\|_{L^{\Awr}((0,\infty ), L^{\Rwr })} \le C \| f \|_{L^{\Bwr '}((0,\infty ), L^{\Rwr '})}.
\end{equation}
Next, using Gagliardo-Nirenberg's inequality
$ \|u\| _{ L^\Rwr }\le C \|\nabla u\| _{ L^2 }^\eta
 \| u\| _{ L^\Gwr }^{1-\eta}$ (for an appropriate value of $\eta\in (0,1)$) and
Sobolev's embedding $H^1 (\R^\Dim ) \subset L^\Gwr (\R^ \Dim ) $, it is not difficult to
show that
\begin{equation} \label{fHolu}
 \|u \| _{ L^\Awr ((0,\infty ), L^\Rwr ) } \le C  \|u\| _{ L^\infty ((0,\infty ), H^1) }^{\frac {\Awr -\Gwr } {\Awr }}  \|u\| _{ L^\Gwr ((0,\infty ), L^\Gwr ) }^{\frac {\Gwr } {\Awr }},
\end{equation}
for all $u\in L^\Gwr ((0,\infty ), L^\Gwr (\R^\Dim )) \cap L^\infty ((0,\infty ), H^1(\R^\Dim ))$.

It follows from~\cite[Theorem~2.1]{HR1} that for every initial value $\DI \in \Inv$ (defined by~\eqref{fDfInv}), the corresponding solution $u$ of~\eqref{SCH}  is global, for both positive and negative time. Thus we may define the flow $(\Sol(t)) _{ t\in \R }$ by
\begin{equation} \label{fFlowu}
\Sol(t) \DI = u(t),
\end{equation}
for $\DI \in \Inv$ and $t\in \R$, where $u$ is the solution of~\eqref{SCH}.
The map $(t, \DI ) \mapsto \Sol (t )\DI $ is continuous $\R \times \Inv \to H^1 (\R^\Dim ) $, where
$\Inv$ is equipped with the $H^1$ topology.  Note that
\begin{equation} \label{fFlowd}
\Sol(- t) \DI = \overline{\Sol (t)  \overline{\DI }  } ,
\end{equation}
for all $t\in \R$ and $\DI \in \Inv$.
It is proved in~\cite[Theorem~2.1]{HR1} that
$\Sol (t) \Inv \subset \Inv$,
for all $t\in \R$, from which it follows by conservation of energy that
\begin{equation} \label{fInvKo}
\Sol (t) \Inv  _{ \omega  }\subset \Inv _{ \omega  },
\end{equation}
for all $t\in \R$ and $\omega \in (0,1)$, where $\Inv  _{ \omega  }$ is defined by~\eqref{fProoft}.

\begin{rem} \label{eRemu}
Let $0<\omega <1$ and $F \subset  \Inv  _{ \omega  }$ be a relatively compact subset of $H^1 (\R^\Dim ) $. It follows that for every  $R, T>0$,
\begin{equation*}
E= \{ \Sol (t) \DI  (\cdot -y);\, \DI  \in F, 0\le t\le T ,  |y|\le R\},
\end{equation*}
is a relatively compact subset of $H^1 (\R^\Dim ) $.
Indeed, suppose $ (\DI  _n) _{n\ge 1 }\subset F$, $(t_n) _{ n\ge 1 }\subset [0,T]$ and $(y_n) _{ n\ge 1 }\subset \{  |y|\le R\}$. By possibly extracting
(and applying Remark~\ref{eRemd}), there exist $\DI  \in \Inv  _{ \omega  }$, $t\in [0,T]$ and $y\in \R^\Dim $ such that $\DI  _n\to \DI  $ in $H^1 (\R^\Dim ) $, $t_n \to t$ and $y_n \to y$ as $n\to \infty $. It  follows that
$\Sol (t_n) \DI  _n(\cdot -y_n) \to \Sol (t) \DI  (\cdot -y)$ in $H^1 (\R^\Dim ) $, which proves the claim.
\end{rem}

The next result is a sufficient condition for scattering.

\begin{prop} \label{eBasic}
Let $\DI  \in H^1 (\R^\Dim ) $ and $u$ be the corresponding solution of~\eqref{SCH}.
If $u$ is global and $u\in L^{\Awr }((0,\infty ), L^{\Rwr }(\R^\Dim ))$, then
$u$ scatters, i.e. there exists a scattering state $u^+\in H^1 (\R^ \Dim ) $ such that~\eqref{fDScat} holds.
\end{prop}

\begin{proof}
It follows from Proposition~2.3 in~\cite{RapDec} that $u\in L^\eta ((0,\infty ), L^\mu  (\R^\Dim ) )$
for every admissible pair $(\eta,\mu )$. Scattering then follows by standard calculations.
(See e.g. Section~7.8 in~\cite{Css}.)
\end{proof}

We next recall a small data global existence property.

\begin{prop} \label{small}
There exist $0< \Dsd \le 1$ and $C$ such that if $ \DI  \in H^1(\R^\Dim ) $ and
$\| \DI \|_{\dot{H}^\Swr } \le \Dsd $, then the corresponding solution $u$ of~\eqref{SCH}
is global for both positive and negative time and
\begin{equation} \label{fsmallu}
 \|u \| _{ L^\Gwr ( \R^{\Dim +1})  } +   \|u \| _{ L^\Awr (\R, L^\Rwr)  } +   \|u\| _{ L^\infty (\R, H^1) } \le
 \Csd   \| \DI  \| _{ H^1 }.
\end{equation}
\end{prop}

\begin{proof}
It follows from Propositions~2.4 and~2.3 in~\cite{RapDec} that if $\DI\in H^1 (\R^\Dim ) $
and $\| \Es{\cdot }  \DI  \|_{ L^\Awr (\R, L^\Rwr) }$ is sufficiently small, then the conclusions of the proposition hold. The result then follows from Strichartz estimate~\eqref{fStrt}.
\end{proof}

\begin{cor} \label{eCritu}
Let $\DI  _1,\DI  _2\in H^1 (\R^\Dim ) $ satisfy $ \|\DI  _1\| _{ H^1 },  \|\DI  _2\| _{ H^1 }\le \Dsd $ where $\Dsd $ is given by Proposition~$\ref{small}$ and let $u_1, u_2$ be the corresponding solutions of~\eqref{SCH}.
If $(t_n^1) _{ n\ge 1 }, (t_n^2) _{ n\ge 1 }\subset \R$ and $(x_n^1) _{ n\ge 1 }, (x_n^2) _{ n\ge 1 }\subset \R^\Dim $ satisfy $ |t_n^1-t_n^2| + |x_n^1-x_n^2| \to \infty $ as $n\to \infty $, then
\begin{equation} \label{eCritu:u}
\sup  _{ t\in \R }  | (u_1(t-t_n^1, \cdot -x_n^1), u_2(t-t_n^2, \cdot -x_n^2))  _{ H^1 }| \goto _{ n\to \infty  }0,
\end{equation}
where $(\cdot ,\cdot ) _{ H^1 }$ is the scalar product in $H^1 (\R^\Dim ) $.
\end{cor}

Before proving Corollary~\ref{eCritu}, we state the following lemma, which we will also use later on. Its elementary proof is left to the reader.

\begin{lem} \label{eRCpu}
Let $\nu \ge 1$, $1\le q<\infty $ and let $E$ be a relatively compact subset of $L^p (\R^\nu ) $.
It follows that
\begin{equation} \label{eRCpu:u}
\sup  _{ u\in E } \int  _{ \{ |x|>R \} }  |u|^q \goto _{ R\to \infty  }0.
\end{equation}
Consequently, if $(y_n^1) _{ n\ge 1 }, (y_n^2) _{ n\ge 1 }\subset \R^\nu $
 and $|y_n^1 - y_n^2| \to \infty$ as $n\to \infty$, then
\begin{equation} \label{eRCpu:d}
\sup  _{ u,v \in E } \int  _{ \R^\nu  }  |u(\cdot -y_n^1)|^{\frac {q} {2}}  |v(\cdot -y_n^2)|^{\frac {q} {2}}
\goto _{ n\to \infty  }0.
\end{equation}
\end{lem}

\begin{proof}[Proof of Corollary~$\ref{eCritu}$]
We note that by Propositions~\ref{small} and~\ref{eBasic} and formula~\eqref{fFlowd}, $u^1$ and $u^2$ scatter at $\pm \infty $.
Suppose by contradiction there exist $\varepsilon >0$ and a sequence $(t_n) _{ n\ge 1 }\subset \R$ such that $ | (u_1(t_n-t_n^1, \cdot -x_n^1), u_2(t_n-t_n^2, \cdot -x_n^2))  _{ H^1 }| \ge \varepsilon $.  Without loss of generality, we may replace $t_n^j$ by $t_n-t_n^j$ for $j=1,2$, so we obtain
\begin{equation} \label{fCrituu}
 | (u_1(t_n^1, \cdot -x_n^1 +x_n^2), u_2(t_n^2, \cdot ))  _{ H^1 }| \ge \varepsilon .
\end{equation}
By possibly extracting, we may assume that one of the following holds: $ |t_n^1|+  |t_n^2|$ is bounded, or  $ |t_n^1|\to \infty $ and $ |t_n^2|\to \infty $, or $ |t_n^1|\to \infty $ and $ |t_n^2|$ is bounded.
In the first case, $u_1(t_n^1)$ and $u_2(t_n^1)$ belong to a relatively compact subset of
$H^1(\R^\Dim )$ and we obtain a contradiction by applying Lemma~\ref{eRCpu}, since $ |x_n^1-x_n^2| \to \infty$.
Suppose now $ t_n^1\to \infty $ and  $ t_n^2\to \pm \infty $.
Since $u_1$ and $u_2$ scatter we may approximate  $u_j(t_n^j)$ by $ \Es{t_n^j} \psi ^j$, $j=1,2$, where $\psi ^j$ is the scattering state of $u^j$, and we deduce from~\eqref{fCrituu} that for $n$ large,
\begin{equation*}
 | ( \Es{(t_n^1-t_n^2)} \psi ^1  (\cdot -x_n^1 +x_n^2),  \psi ^2)  _{ H^1 }|
\ge \frac {\varepsilon } {2}.
\end{equation*}
Since $ |t_n^1-t_n^2| + |x_n^1-x_n^2| \to \infty $, we see that  $\Es{(t_n^1-t_n^2)} \psi ^1  (\cdot -x_n^1 +x_n^2) \rightharpoonup 0$, and we obtain again a contradiction.
Finally,  $ |t_n^2|$ is bounded and, say, $ t_n^1\to \infty $ a similar argument yields a contradiction.
\end{proof}

We now construct the wave operator at $-\infty $ on a certain subset of $H^1 (\R^\Dim ) $.

\begin{prop} \label{WaveO}
Let $0<\omega <1$ and $\Inv  _{ \omega  }$ be defined by~\eqref{fProoft}.
If $\psi  \in H^1(\R^\Dim )$ satisfies
\begin{equation} \label{WaveO:u}
\frac{1}{2 }  \| \nabla  \psi \|_{L^2}^2 M(\psi
)^\Siwr  \le  \omega  E(Q) M(Q)^\Siwr,
\end{equation}
then there exists $\DI  \in \Inv _{ \omega  }$ such that
\begin{equation*}
\| \Sol (-t)\DI   - \Esm{t}  \psi  \|_{ H^1 } \goto _{ t\to  \infty  }0.
\end{equation*}
\end{prop}

\begin{proof}
Given any $\psi \in H^1 (\R^\Dim ) $, is well-known~\cite[Theorem~17]{Straussd}
that there exist $T\in \R$ and a sulution $u\in C((-\infty ,T], H^1 (\R^\Dim ) )$ of~\eqref{SCH} such that
\begin{equation} \label{fWaveOu}
 \|u(-t) -  \Esm{t}  \psi  \|_{ H^1 } \goto _{ t\to  \infty  }0.
\end{equation}
It follows from~\eqref{fWaveOu}, \eqref{WaveO:u},  and~\eqref{Q} that
\begin{equation*}
\begin{split}
\lim_{t\to\infty}\| \nabla  u(-t) \|_{L^2}^2 \| u(-t) \|_{L^2}^{2\Siwr}
& = \| \nabla  \psi \|_{L^2}^2 \| \psi \|_{L^2}^{2 \Siwr} \le
2\omega  E(Q) M(Q)^{\Siwr}\\
& = \omega \frac {\Dim \Pmu -4} {\Dim \Pmu} \| \nabla  Q \|_{L^2}^{2 }\| Q \|_{L^2}^{2\Siwr}.
\end{split}
\end{equation*}
Therefore, if $t$ is sufficiently large, then
\begin{equation} \label{fWaveOt}
\| \nabla  u(-t) \|_{L^2} \| u(-t) \|_{L^2}^{\Siwr} < \| \nabla  Q \|_{L^2} \| Q \|_{L^2}^{\Siwr}.
\end{equation}
It also follows from~\eqref{fWaveOu} that
$\| u(-t) - \Esm{t}  \psi \|_{L^{\Rwr }}\to 0$. Since $ \|\Esm{t} \psi \| _{ L^\Rwr }\to 0$
(see e.g. Corollary~2.3.7 in~\cite{Css}), we see that
$ \|u(-t) \| _{ L^\Rwr }\to 0$, so that
\begin{equation*}
E(u(-t)) \goto  _{ t\to \infty  } \frac{1}{2} \| \nabla \psi  \|_{L^2}^2.
\end{equation*}
Nota also that $M(u(-t)) \to \| \psi \|_{L^2}^2$, so that
\begin{equation} \label{fWaveOq}
E(u(-t)) M(u(-t))^\Siwr \goto _{ t\to \infty  } \frac{1}{2} \| \nabla \psi  \|_{L^2}^2 \| \psi \|_{L^2}^{2\Siwr}
 \le  \omega  E(Q) M(Q)^\Siwr,
\end{equation}
where we used~\eqref{WaveO:u} in the last inequality.
\eqref{fWaveOt}, \eqref{fWaveOq} and conservation of mass and energy
imply that $u(-t) \in \Inv _{ \omega  }$  for all $t\le T$.
In particular, the solution $u$ is global and, since $\Inv _{ \omega  }$ is invariant by the flow $\Sol (t)$ (see~\eqref{fInvKo}), $u(t)\in \Inv  _{ \omega  }$ for all $t\in \R$.
The result follows by setting $\DI  =u(0)$.
\end{proof}

Finally, we prove a perturbation result. It is analogous to
Theorem~2.14 in~\cite{KenigM} (for the energy-critical equation) and
Proposition~2.3 in~\cite{HR2} (for the 3D cubic equation).
The proofs of the above results would apply with obvious modifications, but we use a slightly more direct argument, based on a Gronwall-type inequality (Lemma~\ref{gronwall}).

\begin{prop}\label{LTPT}
Given any  $A\ge 0$, there exist  $\varepsilon (A)>0$ and $C(A)>0$ with the
following property. If  $u\in C([0,\infty ), H^1 (\R^\Dim ))$ is a solution of~\eqref{SCH},
if $ \widetilde{u} \in C([0,\infty ), H^1 (\R^\Dim ))$ and
 $e\in L^1 _{\mathrm {loc}}([0,\infty ), H ^{ -1 } (\R^\Dim ))$ satisfy
\begin{equation*}
i\widetilde{u}_t + \Delta \widetilde{u} + | \widetilde{u}|^{p - 1} \widetilde{u}
= e,
\end{equation*}
for a.a.  $t>0$, and if
\begin{equation} \label{fHypz}
\begin{split}
& \| \widetilde{u} \|_{L^{\Awr}([0,\infty), L^{\Rwr })} \le A,\quad  \| e \|_{L^{\Bwr '}([0,\infty), L^{\Rwr '})} \le
\varepsilon \le \varepsilon (A),\\
& \| e^{i \cdot \Delta } \left(u(0) - \widetilde{u}(0)\right) \|_{L^{\Awr}([0,\infty), L^{\Rwr })}
\le \varepsilon \le \varepsilon (A),
\end{split}
\end{equation}
 then $u\in L^\Awr ((0,\infty ), L^\Rwr (\R^\Dim ) )$ and
  $ \| u - \widetilde{u} \|_{L^{\Awr}([0,\infty), L^{\Rwr })} \le C\varepsilon $.
\end{prop}

\begin{proof}
We let  $w= u-  \widetilde{u} $, so that
\begin{equation} \label{feqw}
iw_t + \Delta  w +  | \widetilde{u} + w |^{\Pmu } (\widetilde{u} + w ) -
| \widetilde{u}|^{\Pmu } \widetilde{u} + e = 0.
\end{equation}
Since
\begin{equation} \label{fLocl}
 \Bigl| | \widetilde{u} + w |^{\Pmu } (\widetilde{u} + w ) - | \widetilde{u}|^{\Pmu } \widetilde{u} \Bigr|
 \le C ( | \widetilde{u} |^{\Pmu } + |w|^{\Pmu } )|w| =C(  | \widetilde{u} |^{\Pmu } |w| +  |w|^{\Pwr }),
\end{equation}
we deduce from the equation~\eqref{feqw}, Strichartz-type estimate
estimate~\eqref{KATO1} and the assumption~\eqref{fHypz}  that there
exists $M>0$ such that
\begin{equation} \label{fPremu}
 \|w\| _{L^{\Awr }((0,t), L ^{ \Rwr  })} \le \varepsilon +
 M  \| \,  | \widetilde{u} |^{\Pmu } |w| +  |w|^{\Pwr } \| _{L^{\Bwr '}((0,t), L ^{ \Rwr  '})}+
 M\varepsilon ,
\end{equation}
for every  $t>0$.
Since $ \| \,  | \widetilde{u} |^{\Pmu } |w|\, \| _{L ^{ \Rwr  '}}\le
 \| \widetilde{u} (t)\| _{ L^{\Rwr } }^{\Pmu }  \|w(t)\| _{ L^{\Rwr } }$,
 we deduce from~\eqref{fPremu} that
\begin{equation} \label{fDepu}
 \|\varphi \| _{ L^{\Awr }(0,t) } \le (M+1)\varepsilon + M  \|\varphi \| _{ L^{\Awr }(0,t) }  ^{ \Pwr }  + M  \| f \varphi  \| _{ L^{\Bwr '} (0,t)},
\end{equation}
where we have set
 \begin{equation*}
 \varphi (t)=    \|w(t)\| _{ L^{\Rwr } },\quad f(t)=   \| \widetilde{u} (t)\| _{ L^{\Rwr } }^{\Pmu } .
 \end{equation*}
 Let
 \begin{equation} \label{fDept}
 \varepsilon (A) \le
2^{-\frac {1} {p-1}} [ (2M+1)\Phi (A^{\Pmu })] ^{- \frac {\Pwr } {\Pmu }} ,
 \end{equation}
 where  $\Phi $ is  given by Lemma~\ref{gronwall}.
Observe that
 \begin{equation}  \label{fDepd}
  \| f \| _{ L^{\frac {(\Pwr ) \Bwr '} {\Pmu  }} (0,\infty )} =  \| f \| _{ L^{ \frac {\Awr  } {\Pmu  }} (0,\infty )} =   \|  \widetilde{u} \| _{ L^{\Awr } ((0,\infty ), L^{\Rwr })}^{\Pmu }
  \le A^{\Pmu }.
 \end{equation}
Given any  $0<T\le \infty $ such that $   \|\varphi \| _{ L^{\Awr
}(0,T) }^{\Pwr} \le \varepsilon \le \varepsilon (A)$, we deduce from~\eqref{fDepu},
\eqref{fDepd} and   Lemma~\ref{gronwall} that  $ \|\varphi \| _{
L^{\Awr }(0,T) } \le (2M+1)\varepsilon  \Phi (A^{\Pmu }) $. Applying~\eqref{fDept}, we see
that $   \|\varphi \| _{ L^{\Awr }(0,T) }^{\Pwr} < \varepsilon ^{\Pwr } \varepsilon (A)^{-\Pmu }
/2\le \varepsilon /2$. It easily follows that we may let $T\to \infty $, so that
$ \|\varphi \| _{
L^{\Awr }(0,\infty ) } \le (2M+1) \Phi (A^{\Pmu }) \varepsilon $.
This is the desired estimate with $C(A)= (2M+1) \Phi (A^{\Pmu }) $.
\end{proof}

\section{Profile decomposition} \label{sProfile}

The following profile decomposition property is an essential ingredient in the proof of
Theorem~\ref{main}.
A quite similar property is established in~\cite{Keraani}, and applied in~\cite{KenigM} to the study of the energy critical NLS.
(An analogous result  for the wave equation is proved in~\cite{BahouriG}.)
A result similar to Theorem~\ref{ProfileE} is proved in~\cite{DHR}, adapted to the 3D cubic NLS.

\begin{thm} \label{ProfileE}
Let $(\phi_n)_{n \ge 1}$ be a bounded sequence of $H^1(\R^\Dim )$.
There is a subsequence, which we
still denote by $(\phi_n)_{n \ge 1}$, and sequences
$(\psi ^j) _{ j\ge 1 } \subset H^1 (\R^\Dim ) $,
$(W_n^j)_{n,j\ge 1} \subset  H^1(\R^\Dim )$,
$(t_n^j) _{ n,j\ge 1 }\subset [0,\infty )$, $( \overline{t}^j) _{ j\ge 1 } \subset [0, \infty ) \cup \{ \infty  \}$
 and $(x_n^j) _{ n,j\ge 1 }  \subset \R^\Dim $
such that for every $\ell \ge 1$
\begin{equation}\label{expan}
\phi_n = \sum_{j=1}^{\ell} \Esm{t_n^j} \psi^j(\cdot - x_n^j) + W_n^ \ell,
\end{equation}
 and
\begin{gather}
t_n^\ell  \goto  _{ n\to \infty  }  \overline{t}^\ell ,   \label{fLimtj}  \\
\| \phi_n \|_{H^\lambda }^2 - \sum_{j = 1}^\ell \| \psi^j \|_{H^\lambda }^2 - \| W_n^\ell
\|_{H^\lambda }^2 \goto _{ n \to \infty  }0, \quad \forall 0\le \lambda \le 1, \label{fAsplitu} \\
E(\phi_n) - \sum_{j=1}^\ell E (\Esm{t_n^j} \psi^j(\cdot - x_n^j)) - E(W_n^\ell ) \goto _{ n\to \infty  }0.
 \label{fEsplit}
\end{gather}
Furthermore, there exists $J\in \N \cup \{\infty \}$ such that $\psi ^j\not = 0$ for all $j<J$ and $\psi ^j=0$ for all $j\ge J$, and
\begin{equation} \label{div}
\lim_{n \to \infty}|t_n^i - t_n^j| + |x_n^i - x_n^j| = \infty ,
\end{equation}
for all $\ell\ge 1$ and $1\le i\not = j<J$.
In addition,
\begin{equation}\label{fAsmall}
\limsup  _{ n\to \infty  } \| \Es{\cdot }  W_n^{\ell } \|_{L^\Awr ((0,\infty ), L^\Rwr )} \goto _{ \ell \to \infty  }0.
\end{equation}
\end{thm}

Theorem~\ref{ProfileE} is proved by iterative application of the following lemma.

\begin{lem}  \label{eAuxd}
Let $a>0$ and let $(v_n) _{ n\ge 1 } \subset H^1 (\R^\Dim ) $ satisfy
\begin{equation} \label{fAuxd:u}
  \limsup _{ n\ge 1 } \|v_n\| _{ H^1 } \le a <\infty .
\end{equation}
If
\begin{equation} \label{fAuxd:d}
\| \Es{\cdot } v_n \|_{L^\infty((0,\infty ), L^\Rwr )} \goto _{ n\to \infty  } A,
\end{equation}
then there exist a subsequence, which we still denote by $(v_n) _{ n\ge 1 }$, and sequences $(t_n) _{ n\ge 1 }\subset [0,\infty )$, $(x_n) _{ n\ge 1 }\subset \R^\Dim $, $\psi \in H^1 (\R^\Dim ) $ and
$(W_n) _{ n\ge 1 }\subset H^1 (\R^\Dim ) $
such that
\begin{equation} \label{fAuxd:t}
v_n= \Esm{t_n} \psi  ( \cdot  -x_n) + W_n,
\end{equation}
with
\begin{equation} \label{fAuxd:qbd}
\Es{t_n} v _n (\cdot +x_n) \weakcv  _{ n\to \infty  }\psi ,
\end{equation}
or, equivalently,
\begin{equation} \label{fAuxd:qb}
\Es{t_n} W_n (\cdot +x_n) \weakcv  _{ n\to \infty  }0,
\end{equation}
in $H^1 (\R^\Dim ) $,
\begin{equation}  \label{fAuxd:q}
 \|v_n\| _{ \dot H^\lambda  }^2 -   \|\psi \| _{ \dot H^\lambda  }^2 -  \|W_n\| _{ \dot H^\lambda  }^2
  \goto _{ n\to \infty  }0,
\end{equation}
for all $0\le \lambda \le 1$ and
\begin{equation}  \label{fAuxd:s}
 \|v_n\| _{L^{\Rwr } }^{\Ppu } -   \|\Esm{t_n} \psi (\cdot - x_n) \|_{L^{\Rwr } }^{\Ppu } -  \|W_n\| _{L^{\Rwr } }^{\Ppu }  \goto _{ n\to \infty  }0.
\end{equation}
Moreover,
\begin{equation}  \label{fAuxd:c}
 \|\psi \| _{H^1 } \ge \nu A^{ \frac {\Dim -2 \Lwr ^2} {2\Lwr (1-\Lwr )} }
 a^{-  \frac {\Dim -2 \Lwr } {2\Lwr (1-\Lwr )}} ,
\end{equation}
where
\begin{equation} \label{fDLwr}
\Lwr = \frac {\Dim \Pmu } {2(\Ppu )} \in  \Bigl( 0, \min \Bigl\{ 1, \frac {\Dim} {2} \Bigr\} \Bigr),
\end{equation}
and the constant $\nu >0$ is independent of  $a$, $A$ and $(v_n) _{ n\ge 1 }$.
Finally, if $A=0$, then for every sequences satisfying~\eqref{fAuxd:t} and~\eqref{fAuxd:qb}, we must have $\psi =0$.
\end{lem}

\begin{proof}
We first introduce a high frequency cut-off.
Fix a real-valued, radially symmetric function $\zeta \in C^\infty  _\Comp (\R^\Dim )$ such that
$0\le \zeta \le 1$, $\zeta (\xi ) =1$ for $ |\xi |\le 1$ and $ \zeta (\xi ) =0$ for $ |\xi |\ge 2$.
Given $r>0$, define $\chi _r\in \Srn$ by $ \widehat{\chi  _r} (\xi )= \zeta (\xi /r)$.
Since $  |\widehat{\chi _r}| \le 1 $, it  is immediate that
\begin{equation}  \label{fAuxu:d}
\| \chi _r \star u\| _{ \dot H^ \lambda  }\le  \|  u \| _{ \dot H^ \lambda  },
\end{equation}
for all $u\in H^1 (\R^\Dim ) $ and $0\le \lambda \le 1$.
Moreover,
\begin{equation*}
\begin{split}
 \|u- \chi  _r \star u\| _{ \dot H^\lambda  }^2 &= \int  _{ \R^\Dim } |\xi |^{2\lambda  }(1-  \widehat{\chi _r} )^2  | \widehat{u} |^2 \\ &  \le
   \int  _{ \{  |\xi |>r \} } |\xi |^{-2(1-\lambda  )}  |\xi |^2  | \widehat{u} |^2
   \le r^{-2(1-\lambda  )} \|\nabla u\| _{ L^2 }^2,
\end{split}
\end{equation*}
so that
 \begin{equation}  \label{fAuxu}
 \|u - \chi _r \star u\| _{ \dot H^\lambda  } \le r^{-(1-\lambda  ) } \|\nabla u\| _{ L^2 },
 \end{equation}
 for all $u\in H^1 (\R^\Dim ) $ and $0\le \lambda \le 1$.
In addition, by Plancherel's formula,
\begin{equation*}
 \chi _r \star u  (0) = \int  _{ \R^\Dim  }\chi _r (x) u(-x)= (-1)^\Dim \int  _{ \R^\Dim  } \widehat {\chi _r} (\xi )  \widehat{ u}(\xi ) .
\end{equation*}
If $\Lwr $ is defined by~\eqref{fDLwr}, then $\Lwr <N/2$, so that
\begin{equation*}
 | \chi _r \star u | (0)\le  \int  _{ \{  |\xi |<2r \} } | \widehat{ u}| \le
 \|  u \| _{ \dot H^ \Lwr   } \Bigl(  \int  _{ \{  |\xi |<2r \} }  |\xi |^{-2\Lwr  } \Bigr)^{\frac {1} {2}}.
\end{equation*}
Thus we see that
\begin{equation} \label{fAuxu:t}
 | \chi _r \star u | (0)\le  \kappa r^{\frac {\Dim -2\Lwr  } {2}} \|  u \| _{ \dot H^ \Lwr   },
\end{equation}
for some constant $\kappa $ independent of $r>0$ and $u\in H^1 (\R^\Dim ) $.
Note also that by Sobolev's embedding
\begin{equation} \label{fSobeu}
 \|u\| _{ L^\Rwr } \le \beta  \|u\| _{ \dot H^\Lwr },
\end{equation}
for some constant $\beta >0$.

If  $A=0$, we let $\psi =0$, $W_n = v _n$, $t_n=0$ and $x_n=0$  for all $n\ge 1$. Properties~\eqref{fAuxd:t}, \eqref{fAuxd:q}, \eqref{fAuxd:s} and~\eqref{fAuxd:c} are immediate.
Furthermore, since $A=0$, it follows in particular that $v _n\to 0$ in $L^\Rwr (\R^\Dim )$
so that~\eqref{fAuxd:qb} holds.

We now suppose $A>0$. Since $\Es{t}$ is an isometry of $\dot H^\Lwr (\R^\Dim )$,
it follows from~ \eqref{fSobeu}, \eqref{fAuxu} and~\eqref{fAuxd:u}  that for $n$ large
\begin{equation*}
 \|\Es{t} v_n - \Es{t} ( \chi _r \star v_n)\| _{ L^\Rwr }\le 2 \beta
 r^{- (1-\Lwr ) }   a \le \frac {A} {2},
\end{equation*}
by choosing
\begin{equation} \label{fDefr}
r=  \Bigl( \frac {4\beta a} {A} \Bigr)^{ \frac  {1} {1-\Lwr }
 } .
\end{equation}
Applying~\eqref{fAuxd:d}, it follows that
\begin{equation} \label{fEstin}
\| \Es{\cdot } (\chi _r \star v_n) \|_{L^\infty((0,\infty ), L^\Rwr)} \ge \frac {A} {4},
\end{equation}
for all sufficiently large $n$.
Note also that (still for $n$ large)
\begin{multline*}
 \|  \Es {\cdot } ( \chi _r \star v_n)\|_{L^\infty((0, \infty ), L^\Rwr)} \\
 \le  \| \Es {\cdot } ( \chi _r \star v_n) \|_{L^\infty((0, \infty ), L^2 )}
 ^{\frac { \Dim -2\Lwr } {\Dim }}
  \| \Es {\cdot } ( \chi _r \star v_n) \|_{L^\infty((0, \infty ), L^\infty )} ^{\frac {2\Lwr } {\Dim }}
 \\  \le (2a) ^{\frac { \Dim -2\Lwr } {\Dim }}
 \| \Es {\cdot } ( \chi _r \star v_n) \|_{L^\infty((0, \infty ), L^\infty )} ^{\frac {2\Lwr } {\Dim }} ,
\end{multline*}
where we used~\eqref{fAuxu:d} in the last inequality.
Thus we deduce from~\eqref{fEstin} that
\begin{equation*}
 \| \Es {\cdot } ( \chi _r \star v_n) \|_{L^\infty( (0,\infty ) , L^\infty )} \ge (2a)
 ^{- \frac {\Dim -2\Lwr } {2\Lwr }}
  \Bigl( \frac {A} {4} \Bigr)^{ \frac {\Dim } {2\Lwr }},
\end{equation*}
for all large $n$. It follows that there exist $(t_n) _{ n\ge 1 }\subset  [0,\infty ) $ and $(x_n) _{ n\ge 1 } \subset \R^\Dim $ such that
\begin{equation} \label{fEstind}
|\Es {t_n } ( \chi _r \star v_n)| (x_n) \ge (4a)
 ^{- \frac {\Dim -2\Lwr } {2\Lwr }}
  \Bigl( \frac {A} {4} \Bigr)^{ \frac {\Dim } {2\Lwr }},
\end{equation}
for all large $n$. Let
\begin{equation*}
w_n(\cdot )= \Es {t_n } v_n (\cdot + x_n).
\end{equation*}
Since $ \|w_n\| _{ H^1 }=  \|v_n\| _{ H^1 }$, it follows from~\eqref{fAuxd:u} that
there exists $\psi \in H^1 (\R^\Dim ) $ such that, after possibly extracting a subsequence,
\begin{equation}  \label{fEstinq}
w_n \weakcv _{ n\to \infty  } \psi ,
\end{equation}
in $H^1 (\R^\Dim ) $.
Since $\Es{t}$ commutes with the convolution with $\chi  _r$, we see that $\Es {t_n } ( \chi _r \star v_n) (x_n)=  ( \chi _r \star w_n) (0)$.
Applying~\eqref{fEstind}, \eqref{fEstinq}  and~\eqref{fAuxu:t}, we obtain
\begin{equation*}
(4a) ^{- \frac {\Dim -2\Lwr } {2\Lwr }}
  \Bigl( \frac {A} {4} \Bigr)^{ \frac {\Dim } {2\Lwr }} \le
   | \chi _r \star \psi |(0) \le \kappa (\Lwr ) r^{\frac {\Dim -2\Lwr } {2}} \|\psi \| _{ \dot H^\Lwr } .
\end{equation*}
Using~\eqref{fDefr}, we deduce that~\eqref{fAuxd:c} holds.
Setting $W_n= v_n - \Esm{t_n} \psi (\cdot -x_n)$, we obtain~\eqref{fAuxd:t}, and
\eqref{fAuxd:qb}  follows from~\eqref{fEstinq}.
We note that by~\eqref{fAuxd:t}, \eqref{fAuxd:qbd} and~\eqref{fAuxd:qb} are equivalent.
Furthermore, given any $0\le \lambda \le 1$, it follows from~\eqref{fAuxd:t} that
\begin{equation*}
\begin{split}
 \|v_n\| _{ \dot H^\lambda  }^2&=   \| \Esm{t_n} \psi (\cdot -x_n) \| _{ \dot H^\lambda  }^2+  \|W_n\| _{ \dot H^\lambda  }^2
 + 2 ( \Esm{t_n} \psi (\cdot -x_n), W_n)  _{ \dot H^\lambda  }\\
 &=   \|\psi  \| _{ \dot H^\lambda  }^2+  \|W_n\| _{ \dot H^\lambda  }^2
 + 2 ( \Esm{t_n} \psi (\cdot -x_n), v_n - \Esm{t_n} \psi (\cdot -x_n))  _{ \dot H^\lambda  }.
\end{split}
\end{equation*}
Applying~\eqref{fEstinq}, we deduce that~\eqref{fAuxd:q} holds.
We next prove~\eqref{fAuxd:s}
and we set
\begin{equation*}
f_n =   \Bigl| \|v_n\| _{L^{\Rwr } }^{\Ppu } -   \|\Esm{t_n} \psi (\cdot - x_n) \|_{L^{\Rwr } }^{\Ppu } -  \|W_n\| _{L^{\Rwr } }^{\Ppu } \Bigr|.
\end{equation*}
We recall that  for every $P>1$ and $\ell \ge 2$ there exists a constant $C_{P,\ell} $ such that
 \begin{equation} \label{fGerd}
  \left| \, \Bigl| \sum_{  j=1 }^\ell z_j  \Bigr|^{P} -\sum_{ j=1 }^\ell  |z_j|^{P} \right| \le  C_{P,\ell } \sum_{ j\not = k }  |z_j| \, |z_k|^{P-1 },
 \end{equation}
  for all $(z_j) _{ 1\le j\le \ell }\subset {\mathbb {C}} ^\ell$.
 (This is inequality~(1.10) in~\cite{Gerard}.)
It follows from~\eqref{fGerd} that there exists a constant $C$ such that
 for all $z_1, z_2\in {\mathbb {C}} $,
 \begin{equation} \label{fGerq}
\Bigl|  |z_1+ z_2|^{\Ppu } -|z_1|^{\Ppu } -|z_2|^{\Ppu }  \Bigr|  \le  C  |z_1| \, |z_2|
(|z_1|^{\Pmu }+  |z_2|^{\Pmu }).
 \end{equation}
We deduce from~\eqref{fAuxd:t} and~\eqref{fGerq} that
 \begin{equation*}
 f_n \le C \int  _{ \R^\Dim  }  |\Es{t_n} \psi | \,  |W_n (\cdot +x_n)| h_n,
 \end{equation*}
where
\begin{equation*}
h_n =  |\Es{t_n} \psi | ^{\Pmu } +  |W_n (\cdot +x_n)| ^{\Pmu }.
\end{equation*}
Note that
\begin{equation*}
 \|h_n \| _{ L^{\frac {\Ppu } {\Pmu }} }\le   C( \| \Es{t_n} \psi  \| _{ L^{\Rwr } }^{\Pmu }+
 \| W_n\| _{ L^{\Rwr } }^{\Pmu }) \le C (  \|\psi \| _{ H^1 } +  \|W_n\| _{ H^1 } )^{\Pmu }
 \le C.
\end{equation*}
Assume by contradiction there exist $\varepsilon >0$ and a sequence $n_k \to \infty $ such that $f _{ n_k }\ge \varepsilon $.
By possibly extracting, we may assume that either $ |t _{ n_n }|\to \infty $ or else $t _{ n_k  }\to  \overline{t} \in \R $. In the first case,
\begin{equation*}
f _{ n_k } \le  C \| \Es{t_{n_k}} \psi  \| _{ L^{\Rwr } } \| W _{n_k} \| _{ L^{\Rwr } }  \|h_{n_k} \| _{ L^{\frac {\Ppu } {\Pmu }} } \le C \| \Es{t_{n_k}} \psi  \| _{ L^{\Rwr } } \goto _{ k\to \infty  }0,
\end{equation*}
since $\psi \in H^1 (\R^N ) $ and $ |t _{ n_k }|\to \infty $. (See e.g. Corollary~2.3.7 in~\cite{Css}.)
This is absurd.
In the second case, it follows from~\eqref{fAuxd:qb} that $W_{n_k} (\cdot  + x_{n_k}) \weakcv 0$ in  $H^1 (\R^\Dim ) $ as $k\to \infty $. By Sobolev's embedding, we deduce that
$W_{n_k} (\cdot  + x_{n_k}) \to 0$ in $L^{\Rwr } (B _R) $ strongly for every $R>0$, where $B_R$ is the ball of $ \R^N $ with center $0$ and radius $R$.
Note also that $\Es{t _{ n_k} } \psi $ belongs to a compact subset of $H^1 (\R^\Dim ) $, hence of
$L^{\Rwr } (\R^\Dim ) $.  Thus for every $\delta >0$, there exists $R$ such that $ \| \Es{t _{ n_k} } \psi  \| _{ L^{\Rwr } (\{  |x|>R\})} <\delta $. Using the boundedness of $\Es{t _{ n_k} } \psi $ and $W _{ n_k }$ in $L^{\Rwr } (\R^\Dim )$ and that of $h _{ n_k }$ in $L^{\frac {\Ppu } {\Pmu }}  (\R^\Dim )$, we estimate
\begin{equation*}
f _{ n_k } \le C  \| \Es{t _{ n_k} } \psi \| _{ L^{\Rwr  }(\{  |x|>R\})}  + C \| W _{ n_k }
 (\cdot +x _{ n_k })\| _{ L^{\Rwr  }(\{  |x|<R\})} .
\end{equation*}
We first choose $R$ large enough so that the first term on the right hand side is smaller than $\varepsilon /4$, then $k_0$ large enough so that the second term is also less than $\varepsilon /4$ for $k\ge k_0$, and we deduce that $f _{ n_k }\le \varepsilon /2$ for $k\ge k_0$. This is absurd and proves~\eqref{fAuxd:s}.
Finally, if $A=0$, then in particular $  \|\Es {t_n } v_n \| _{ L^\Rwr } \to 0$.
Thus $\psi =0$ by~\eqref{fAuxd:qbd}. This shows the last statement of the lemma and completes the proof.
\end{proof}

Before proceeding to the proof of Theorem~\ref{ProfileE}, we prove the following property.
\begin{lem}  \label{eAuxt}
Let $(t_n) _{ n\ge 1 }\subset \R$ and $(x_n) _{ n\ge 1 } \subset \R^\Dim $
satisfy
\begin{equation} \label{fAuxt:d}
 |t_n| + |x_n|\goto  _{ n\to \infty  }\infty .
\end{equation}
It follows that
\begin{equation} \label{fAuxt:u}
\Es{t_n} \psi (\cdot +x_n) \weakcv  _{ n\to \infty  }0,
\end{equation}
in $H^1 (\R^\Dim ) $ for all $\psi  \in H^1 (\R^\Dim ) $.
Moreover, if $(z_n) _{ n\ge 1 }\subset H^1 (\R^\Dim ) $ and $\psi \in H^1 (\R^\Dim ) $ satisfy
\begin{equation}  \label{fAuxt:t}
z_n \weakcv  _{ n\to \infty  }0, \quad \Es{t_n} z_n (\cdot +x_n) \weakcv  _{ n\to \infty  } \psi ,
\end{equation}
in $H^1 (\R^\Dim ) $ and if $\psi \not = 0$, then~\eqref{fAuxt:d} holds.
\end{lem}

\begin{proof}
We prove the first statement, so we assume~\eqref{fAuxt:d}.
By a standard density argument, we need only show that for every $\psi,\zeta   \in \Srn$, $( \Es{t _{ n }} \psi (\cdot +x _{ n }) , \zeta  ) _{ H^1 }\to 0$ as $n\to \infty $.
Assume by contradiction that there exist $\psi , \zeta  \in \Srn$, $\varepsilon >0$ and sequences
$(t _{ n_k }) _{ k\ge 1 }$ and $(x _{ n_k }) _{ k\ge 1 }$ satisfying~\eqref{fAuxt:d}
 such that $|( \Es{t _{ n_k }} \psi (\cdot +x _{ n_k }) , \zeta  ) _{ H^1 } |\ge \varepsilon $. By considering a subsequence, we may assume that either $t _{ n_k }\to \infty $ or else $t _{ n_k }$ is bounded. In the first case,
\begin{equation*}
\begin{split}
|( \Es{t _{ n_k }} \psi (\cdot +x _{ n_k }) , \zeta  ) _{ H^1 }|  &= |( \psi , \Esm{t _{ n_k }} \zeta  (\cdot -x _{ n_k }) ) _{ H^1 }| \\ &
 \le  |t _{ n_k }| ^{ -\frac {\Dim } {2} } \|\psi \| _{ W^{1,1} }  \|\zeta  \| _{ W^{1,1} } \goto _{ k\to \infty  }0,
\end{split}
\end{equation*}
which is absurd. In the second case,   $ |x _{ n_k }|\to \infty $, so that $\zeta  (\cdot -x _{ n_k } )\weakcv 0$ in $H^1 (\R^\Dim ) $. Since $ \Es{t _{ n_k }} \psi$ belongs to a compact subset of $H^1 (\R^\Dim ) $, it follows (see Lemma~\ref{eRCpu}) that
\begin{equation*}
|( \Es{t _{ n_k }} \psi (\cdot +x _{ n_k }) , \zeta  ) _{ H^1 }| =
|( \Es{t _{ n_k }} \psi  , \zeta  (\cdot - x _{ n_k })) _{ H^1 }| \goto _{ k\to \infty  }0,
\end{equation*}
which is also absurd.
We now prove the second statement, so we assume~\eqref{fAuxt:t}.
Suppose by contradiction that $\psi \not = 0$ and there exist $n_k\to \infty $ such that $ |t _{ n_k }| +  |x _{ n_k }|$ is bounded. By considering a subsequence, we may assume $t _{ n_k } \to  \overline{t} $ and
$x _{ n_k } \to  \overline{x} $. Since $z_n \weakcv 0$, it follows easily that $\Es{t_n} z_n (\cdot +x_n) \weakcv  0$, which is absurd.
\end{proof}

\begin{proof}[Proof of Theorem~$\ref{ProfileE}$]
We set
\begin{equation*}
a = \limsup _{ n\to \infty  } \|\phi _n\| _{ H^1 },
\end{equation*}
we let
\begin{equation*}
W^0_n= \phi _n
\end{equation*}
and we construct by induction on $\ell $ the various sequences so that
for every $1\le j\le \ell$,
\begin{equation} \label{fRect}
W_n^{j -1} = \Esm{ t_n^j } \psi ^j  (\cdot - x_n^j ) + W_n ^j ,
\end{equation}
for all $n\ge 1$,
\begin{equation} \label{fLimtjd}
t_n^j \goto  _{ n\to \infty  } \overline{t}^j,
\end{equation}
\begin{equation} \label{fRecq}
\Es{t_n^j } W_n ^{j-1}  (\cdot +x_n^j ) \weakcv  _{ n\to \infty  } \psi ^j , \quad
\Es{t_n^j } W_n ^j  (\cdot +x_n^j ) \weakcv  _{ n\to \infty  }0,
\end{equation}
in $H^1 (\R^\Dim ) $,
\begin{equation} \label{fRecc}
 \|W_n^{j -1}\| _{ \dot H^\lambda  }^2 -   \|\psi ^j  \| _{ \dot H^\lambda  }^2 -  \|W_n^j \| _{ \dot H^\lambda  }^2
  \goto _{ n\to \infty  }0,
\end{equation}
for all $0\le \lambda \le 1$,
\begin{equation} \label{fReccbu}
 \|W_n^{j -1}\| _{ L^{\Rwr } }^{\Ppu }  -   \|\Esm{t_n^j} \psi^j(\cdot - x_n^j)  \| _{ L^{\Rwr } }^{\Ppu }  -  \|W_n^j \| _{ L^{\Rwr } }^{\Ppu }   \goto _{ n\to \infty  }0,
 \end{equation}
 and
\begin{gather}
\|\Es {\cdot } W _n^{j -1}\| _{ L^\infty ((0,\infty ), L^\Rwr ) } \goto _{ n\to \infty  }A_j,
\label{fIndlim} \\
 \|\psi ^j \| _{ H^1 } \ge \nu a^{- \frac {\Dim -2\Lwr } {2\Lwr (1-\Lwr )}}
 A_j ^{ \frac {\Dim -2\Lwr ^2} {2\Lwr (1-\Lwr )}} ,\label{fRecs}
\end{gather}
where $\nu$ is the constant in Lemma~\ref{eAuxd}.

For $\ell =1$, we set
\begin{equation*}
A_1= \limsup  _{ n\to \infty  } \|\Es {\cdot } \phi _n\| _{ L^\infty ((0,\infty ), L^\Rwr ) }.
\end{equation*}
We extract a subsequence so that
\begin{equation*}
\|\Es {\cdot } \phi _n\| _{ L^\infty ((0,\infty ), L^\Rwr ) } \goto _{ n\to \infty  }A_1,
\end{equation*}
and we apply Lemma~\ref{eAuxd} with $v_n= \phi _n= W_n^0$.
Properties~\eqref{fRect}, \eqref{fRecq}--\eqref{fReccbu} and~\eqref{fRecs} are immediate consequences of Lemma~\ref{eAuxd}.
Moreover, by possibly extracting, we may assume that~\eqref{fLimtjd} holds for some $0\le  \overline{t}^j \le \infty  $.

Given $\ell \ge 2$, suppose $t^j_n,  \overline{t}^j , x_n^j, \psi ^j, W^j_n$ have been constructed for all $n\ge 1$ and $j\le \ell -1$, and set
\begin{equation*}
A_\ell = \limsup  _{ n\to \infty  } \|\Es {\cdot } W _n^{\ell -1}\| _{ L^\infty ((0,\infty ), L^\Rwr ) }.
\end{equation*}
We extract a subsequence so that~\eqref{fIndlim} holds for $j=\ell$
and we apply Lemma~\ref{eAuxd} with $v_n= W _n^{\ell -1}$.
We obtain (after possible extraction)
$(t_n^\ell ) _{ n\ge 1 }\subset (0,\infty )$, $0\le  \overline{t}^\ell \le \infty $, $(x_n^\ell ) _{ n\ge 1 }\subset \R^\Dim $, $\psi ^\ell \in H^1 (\R^\Dim ) $ and  $(W_n^\ell ) _{ n\ge 1 }\subset H^1 (\R^\Dim ) $
such that~\eqref{fRect}--\eqref{fReccbu} hold for $j=\ell $ and
\begin{equation}  \label{fLimspr}
 \|\psi ^\ell \| _{ H^1 } \ge \nu A_\ell ^{ \frac {\Dim -2\Lwr ^2} {2\Lwr (1-\Lwr )}}
(\limsup _{ n\to \infty  } \|W_n ^\ell \| _{ H^1 })^{- \frac {\Dim -2\Lwr } {2\Lwr (1-\Lwr )}} .
\end{equation}
Summing up~\eqref{fRecc} in $j $ from $0$ to $\ell $, we obtain~\eqref{fAsplitu} at rank $\ell $, from which it follows that
\begin{equation*}
\limsup _{ n\to \infty  } \|W_n ^\ell \| _{ H^1 }\le \limsup _{ n\to \infty  } \|\phi _n\| _{ H^1 }\le a;
\end{equation*}
and so~\eqref{fLimspr} yields~\eqref{fRecs}.

We note that at every iteration, we extract subsequences from the previously constructed sequences $t_n^j, x_n^j$ and $W_n^j$.
However, this does not affect the identity~\eqref{fRect}, nor the
 limits~\eqref{fLimtjd}, \eqref{fRecq}, \eqref{fRecc} and~\eqref{fIndlim} (and therefore the estimate~\eqref{fRecs}).
Moreover we may, at the $\ell$-th iteration, extract so that
$t_n^j, x_n^j$ and $W_n^j$ are unchanged for $j \le \ell-1$ and $n\le \ell$.
In this way, the above iteration constructs the sequences $\psi ^j$, $t_n^j$, $ x_n^j$ and $W_n^j$
for all $j\ge 1$ and $n\ge 1$ and all properties~\eqref{fRect}--\eqref{fRecs} are satisfied.

We now show that the sequences that we constructed satisfy all the conclusions.
As observed above, \eqref{fAsplitu} follows by summing~\eqref{fRecc} in $j$ from $0$ to $\ell $.
Similarly, \eqref{expan} follows by summing~\eqref{fRect}.
Note that by~\eqref{fAsplitu}
$ \sum_{j = 1}^\ell \| \psi^j \|_{H^1 }^2 \le a$ so that, letting $\ell \to \infty $,
$ \sum_{j = 1}^\infty  \| \psi^j \|_{H^1 }^2 \le a$.
Applying~\eqref{fRecs}, we deduce that
\begin{equation*}
 \sum_{j = 1}^\infty
A_j ^{ \frac {\Dim -2\Lwr ^2} {\Lwr (1-\Lwr )}}  \le \frac {1} {\nu ^2}
  a^{ \frac {\Dim -\Lwr -\Lwr ^2} {\Lwr (1-\Lwr )}} <\infty ,
\end{equation*}
so that
\begin{equation} \label{fFinu}
A_\ell \goto  _{ \ell \to \infty  }0.
\end{equation}
Note that
$ \| f \| _{ L^\Awr } \le  \| f \| _{ L^\Qwr }^\theta  \|f\| _{ L^\infty  }^{1-\theta }$
with $\theta = \frac {2[4- (\Dim -2) \Pmu ]} {\Dim \Pmu^2}\in (0,1)$; and so
\begin{equation}  \label{fFind}
\begin{split}
 \| \Es {\cdot } W _n^{\ell } \| _{ L^\Awr ((0,\infty ), L^{\Rwr } )}
& \le  \| \Es {\cdot } W _n^{\ell } \| _{ L^\Qwr ((0,\infty ), L^{\Rwr } )}^\theta
 \| \Es {\cdot } W _n^{\ell } \| _{ L^\infty  ((0,\infty ), L^{\Rwr }) }^{1-\theta } \\
 & \le C  \|  W _n^{\ell } \| _{ H^1}^\theta
 \| \Es {\cdot } W _n^{\ell } \| _{ L^\infty  ((0,\infty ), L^{\Rwr }) }^{1-\theta },
\end{split}
\end{equation}
by Strichartz estimate~\eqref{fStru}. Since $\limsup  _{ n\to \infty  }  \| W_n^\ell \| _{ H^1 } \le a$ by~\eqref{fAsplitu},  we deduce from estimates~\eqref{fFind} and~\eqref{fFinu}
that~\eqref{fAsmall} holds.
The last statement of Lemma~\ref{eAuxd} shows that if $\psi ^j=0$ for some $j\ge 1$, then $\psi ^\ell =0$ for all $\ell \ge j$; and so there exists $J\in \N \cup \{\infty \}$ such that $\psi ^j\not = 0$ for all $j<J$ and $\psi ^j=0$ for all $j\ge J$.

We next prove~\eqref{div}.
We suppose $J\ge 3$ (otherwise there is nothing to be proved) and we argue by induction.
We note that by~\eqref{fRecq},
\begin{equation*}
\Es{t_n^1} W_n ^{1} (\cdot +x_n^1) \weakcv  _{ n\to \infty  } 0,
\quad \Es{t_n^2} W_n ^{1} (\cdot +x_n^2) \weakcv  _{ n\to \infty  } \psi ^2.
\end{equation*}
Applying Lemma~\ref{eAuxt} (second statement) with $z_n= \Es{t_n^1} W_n ^{1} (\cdot +x_n^1)$,
we deduce that $ |t_n^2-t_n^1| +  |x_n^2-x_n^1| \to \infty $.
Suppose now $J\ge 4$ and~\eqref{div} has been proved for all $1\le j\not = k\le \ell -1$ for some $\ell <J$. Given $1\le j\le \ell -1$, it follows from~\eqref{fRect} that
\begin{equation*}
W_n^{\ell -1}- W_n^{j-1}= - \sum_{ k= j }^{\ell -1}
\Esm{ t_n^k} \psi ^k (\cdot - x_n^k),
\end{equation*}
so that
\begin{multline} \label{fDivu}
\Es{t_n^j}W_n^{\ell -1}(\cdot +x_n^j)-  \Es{t_n^j}W_n^{j-1} (\cdot +x_n^j) \\
=- \psi ^j - \sum_{ k= j+1 }^{\ell -1}
\Es{(t_n^j - t_n^k)} \psi ^k (\cdot + x_n^j -x_n^k).
\end{multline}
Applying~\eqref{div}, it follows from Lemma~\ref{eAuxt} (first statement) that the right-hand side of~\eqref{fDivu} weakly converges to $- \psi ^j$. Using also~\eqref{fRecq}, we deduce from~\eqref{fDivu} that
\begin{equation*}
\Es{t_n^j}W_n^{\ell -1}(\cdot +x_n^j)
\weakcv  _{ n\to \infty  } 0.
\end{equation*}
On the other hand, it follows from~\eqref{fRecq} (with $j=\ell $) that
\begin{equation*}
\Es{t_n^\ell} W_n ^{\ell-1} (\cdot +x_n^\ell ) \weakcv  _{ n\to \infty  } \psi ^\ell.
\end{equation*}
Applying Lemma~\ref{eAuxt} (second statement), we deduce as above that  $ |t_n^\ell-t_n^j| +  |x_n^\ell-x_n^j| \to \infty $.
Thus we see that~\eqref{div} holds for all $1\le j\not = k\le \ell$.

Finally, we deduce from~\eqref{fRecc} (with $\lambda =1$) and~\eqref{fReccbu} that
\begin{equation*}
 E(W_n^{j -1})  -   E(\Esm{t_n^j} \psi^j(\cdot - x_n^j) ) -  E(W_n^j )    \goto _{ n\to \infty  }0,
 \end{equation*}
and~\eqref{fEsplit} follows by summing up the above estimate from $j=1$ to $j=\ell$.
\end{proof}

\section{Existence of a critical solution} \label{sCritSol}

This section is devoted to the following proposition, which is an essential step in the proof of Theorem~\ref{main}.

\begin{prop}\label{critical}
Let $\Nonad $ be defined by~\eqref{fProofd}, $\Inv  _{ \omega  }$ be defined by~\eqref{fProoft},
and let $\omega _0\in (0,1]$ be defined by~\eqref{fProofq}. If $\omega _0<1$, then there exists
$\Fcr \in \Inv$ such that $\Fcr \not \in \Nonad$.
Moreover, if $\Ucr $ is the corresponding solution of~\eqref{SCH} then
 there exists a function $x \in C(\R, \R^\Dim )$ such that
$\{ \Ucr (t, \cdot - x(t));\, t\ge 0\}$
is relatively compact in $H^1(\R^\Dim )$.
\end{prop}

In the proof Proposition~\ref{critical}, we use the following lemma, in which we construct  nonlinear profiles associated to certain elements of $\Inv  _{ \omega  }$.

\begin{lem} \label{eNLP}
Let $0<\omega <1$  and $(t_n) _{ n\ge 1 }\subset (0,\infty )$. If $\Esm{t_n} \psi \in \Inv  _{ \omega  }$ for all $n\ge 1$ and $t_n\to  \overline{t} \in [0,\infty ] $, then there exists a ``nonlinear profile" $ \widetilde{\psi }  \in \Inv  _{ \omega  }$ such that
\begin{gather}
 \| \widetilde{\psi } \| _{ L^2 } =  \|\psi \| _{ L^2 }, \label{eNLP:u} \\
E(  \widetilde{\psi })= \lim  _{ n\to \infty  } E( \Esm{t_n} \psi ),   \label{eNLP:d} \\
 \| \Sol (-t)  \widetilde{\psi } - \Esm{t} \psi \| _{ H^1 } \goto  _{ t\to  \overline{t}   }0 \label{eNLP:t}.
\end{gather}
\end{lem}

\begin{proof}
If $ \overline{t} = \infty $, then $ \| \Esm{t_n} \psi \| _{ L^\Rwr } \to 0$ as $n\to \infty $, so that
\begin{equation*}
E(\Esm{t_n} \psi) - \frac {1} {2}  \|\nabla   \psi \| _{ L^2 }^2
= E(\Esm{t_n} \psi ) - \frac {1} {2}  \|\nabla \Esm{t_n } \psi \| _{ L^2 }^2
\goto  _{ n\to \infty  }0.
\end{equation*}
Since $\psi \in \Inv  _{ \omega  }$, we deduce that
\begin{equation} \label{fNLPu}
\frac {1} {2} \|\nabla \psi \| _{ L^2 }^2 M(\psi )^\Siwr
\le  \ \omega  E(Q) M(Q)^\Siwr.
\end{equation}
Thus we may apply Proposition~\ref{WaveO}  and obtain $ \widetilde{\psi }  \in \Inv  _{ \omega  }$ such that~\eqref{eNLP:t} holds.
If $ \overline{t} <\infty $, we set
$ \widetilde{\psi } = \Sol ( \overline{t}) [ \Esm {\overline{t}} \psi ]$.
Since $\Esm {\overline{t}} \psi  \in \Inv  _{ \omega  }$, we deduce from~\eqref{fInvKo} that $ \widetilde{\psi }  \in \Inv  _{ \omega  }$, and~\eqref{eNLP:t} follows from the continuity of the flow.
Finally, \eqref{eNLP:t} together with conservation of charge (for both the linear and nonlinear flows) and energy (for the nonlinear flow) yield~\eqref{eNLP:u} and~\eqref{eNLP:d}.
\end{proof}

The main step in the proof of Proposition~\ref{critical} is the following lemma, which says that, under appropriate assumptions, the profile decomposition of Theorem~\ref{ProfileE} contains at most one nonzero element.

\begin{lem} \label{ePrinu}
Suppose $\omega _0<1$. Let $(\omega _n) _{ n\ge 1 }\subset (0,1)$ satisfy $\omega _n \to \omega _0$ as $n\to \infty $. Let $(\phi _n) _{ n\ge 1 } \subset H^1 (\R^\Dim ) $ and suppose
$\phi _n \in \Inv  _{ \omega _n }$ and $M(\phi _n)=1$, $\phi _n\not \in \Nonad$  for all $n\ge 1$.
It follows that there exist a subsequence, which we still denote by $(\phi _n ) _{ n\ge 1 }$,
$\psi \in H^1 (\R^\Dim ) $, $(W_n) _{ n\ge 1 } \subset H^1 (\R^\Dim ) $,
$(x_n) _{ n\ge 1 }\subset \R^\Dim $, $(\tau _n) _{ n\ge 1 }\subset [0,\infty )$ and $0\le  \overline{\tau }\le \infty  $ such that $M(\psi )=1$ and
\begin{gather}
\phi _n= \Esm{\tau _n} \psi (\cdot -x_n) + W_n,\quad n\ge 1\label{ePrinu:u} \\
\tau _n \goto _{ n\to \infty  } \overline{\tau }, \label{ePrinu:d}  \\
 \|W_n\| _{ H^1 } \goto _{ n\to \infty  }0. \label{ePrinu:t}
\end{gather}
\end{lem}

\begin{proof}
Since $\phi _n \in \Inv  _{ \omega _n }$ and $M(\phi _n)=1$, we deduce from~\eqref{fCdeltad} that
\begin{equation}  \label{fPrinu}
\|\nabla \phi _n\| _{ L^2 }   \le\omega _n^{ \frac {1} {2} }  \|\nabla Q\| _{ L^2 }  \|Q\| _{ L^2 }^\Siwr.
\end{equation}
Applying Theorem~\ref{ProfileE}  to $ (\phi _n) _{ n\ge 1 }$, we write
(after extracting a subsequence)
\begin{equation} \label{fCritq}
\phi_n = \sum_{j=1}^{\ell} \Esm{t_n^j} \psi^j(\cdot - x_n^j) + W_n^ \ell,
\end{equation}
 where the various sequences satisfy properties~\eqref{fLimtj} through~\eqref{fAsmall}.
 Since $M(\phi _n)=1$,
it follows in particular from~\eqref{fAsplitu}  that for every $\ell \ge 1$,
$\sum_{j = 1}^\ell M( \psi^j ) \le 1$. Thus
\begin{equation} \label{fCritp}
 \sum_{j = 1}^\infty  M( \psi^j ) \le 1,
\end{equation}
and we set
\begin{equation} \label{fCritpbu}
 \widetilde{M}= \sup _{ j\ge 1 } M( \psi^j ) \le 1.
\end{equation}
Similarly, we deduce from~\eqref{fAsplitu} and~\eqref{fPrinu}  that
\begin{equation*}
 \sum_{j = 1}^\ell  \|\nabla \psi^j \| _{ L^2 }^2 \le \limsup _{ n\to \infty  } \|\nabla \phi _n\| _{ L^2 }^2
 \le \omega _0  \|\nabla Q\| _{ L^2 }^2  \|Q\| _{ L^2 }^{2 \Siwr} ,
\end{equation*}
so that
\begin{equation}  \label{fCritn}
 \sum_{j = 1}^\infty   \|\nabla \psi^j \| _{ L^2 }^2 \le  \omega _0  \|\nabla Q\| _{ L^2 }^2  \|Q\| _{ L^2 }^{2 \Siwr} .
\end{equation}
Applying~\eqref{fCritp} and~\eqref{fCritn}, we see that  for every $j,n\ge 1$
\begin{equation} \label{fCrituz}
  \|\nabla \Esm{t_n^j} \psi^j  \| _{ L^2 }  \| \Esm{t_n^j} \psi^j \| _{ L^2 }^{ \Siwr}  =
  \|\nabla \psi ^j\| _{ L^2 }  \|\psi ^j\| _{ L^2 }^{ \Siwr}  \le
   \omega _0  ^{\frac {1} {2}}  \|\nabla Q\| _{ L^2 }  \|Q\| _{ L^2 }^{ \Siwr}
\end{equation}
It follows from~\eqref{fCrituz} and~\eqref{fCdeltau} that
\begin{equation}  \label{fCrituubu}
E(\Esm{t_n^j} \psi^j) \ge  \frac {\Dim \Pmu -4} {2 \Dim \Pmu} \|\nabla \psi ^j\| _{ L^2 }^2 \ge 0,
\end{equation}
for all $n,j\ge 1$.
Similarly as above, we deduce from~\eqref{fAsplitu} (with $\lambda =0$ and $\lambda =1$) that, given any $\ell\ge 1$,
\begin{equation*}
\limsup _{ n\to \infty  }  \|W_n ^\ell \| _{ L^2 }\le 1,
\quad \limsup _{ n\to \infty  }  \|\nabla W_n ^\ell \| _{ L^2 }^2 \le \omega _0  \|\nabla Q\| _{ L^2 }^2  \|Q\| _{ L^2 }^{2 \Siwr},
\end{equation*}
so that
\begin{equation}  \label{fCritud}
E(W_n^\ell ) \ge  \frac {\Dim \Pmu -4} {2 \Dim \Pmu} \|\nabla W_n^\ell \| _{ L^2 }^2 \ge 0,
\end{equation}
for all sufficiently large $n$ (depending on $\ell$).
Since $\phi _n\in \Inv  _{ \omega _n }$ and $M(\phi _n)=1$, it follows from~\eqref{fEsplit} and~\eqref{fCritud} that, given any $j\ge 1$,
\begin{equation} \label{fYAM}
\limsup _{ n\to \infty  } E (\Esm{t_n^j} \psi^j) \le \limsup _{ n\to \infty  } E(\phi _n) \le \omega _0
E(Q) M(Q)^\Siwr,
\end{equation}
and we set
\begin{equation} \label{fSupld}
 \widetilde{E}= \sup _{ j\ge 1 } \Bigl[  \limsup _{ n\to \infty  } E (\Esm{t_n^j} \psi^j) \Bigr]
 \le \omega _0 E(Q) M(Q)^\Siwr.
\end{equation}
Given any $\omega _0<\omega <1$, it follows from~\eqref{fCritpbu} and~\eqref{fSupld} that
for every $j\ge 1$, $\Esm{t_n^j} \psi^j\in  \Inv _{ \omega  }$ for all large $n$.
We now apply Lemma~\ref{eNLP} and obtain the nonlinear profiles
\begin{equation} \label{fNLProfd}
\widetilde{\psi }^j  \in \Inv  _{ \omega  },
\end{equation}
such that
\begin{equation} \label{fNLProfu}
 \| \Sol (-t)  \widetilde{\psi }^j - \Esm{t} \psi^j \| _{ H^1 } \goto  _{ t\to  \overline{t}^j   }0.
\end{equation}
We set
\begin{equation} \label{fDefoz}
 \widetilde{\omega }= \frac { \widetilde{E}  \widetilde{M}^\Siwr  } {E(Q) M(Q)^\Siwr}\le \omega _0.
\end{equation}
 Note that by~\eqref{fNLProfd}, \eqref{fCritpbu}, \eqref{fSupld} and~\eqref{eNLP:d},
\begin{equation} \label{fFinubu}
 \widetilde{\psi }^j \in \Inv  _{  \widetilde{\omega }  },\quad j\ge 1.
\end{equation}
We now prove by contradiction that
\begin{equation} \label{fAbsu}
 \widetilde{\omega } = \omega _0.
\end{equation}
The idea of the proof is now to approximate
\begin{equation} \label{fAppru}
\Sol (t) \phi _n \approx  \sum_{ j=1 }^\ell  \Sol (t-t_n^j)  \widetilde{\psi }^j (\cdot -x_n^j)
\end{equation}
If~\eqref{fAbsu} fails, then $\widetilde{\omega } < \omega _0$, so  all the terms on the right hand side of the approximation~\eqref{fAppru} scatter.
Furthermore, for $\ell $ and $n$ sufficiently large, it follows from the divergence property~\eqref{div} that the remainder in~\eqref{fAppru} converges to $0$ as $t\to \infty $. It follows that $\Sol (t) \phi _n$ scatters, which yields a contradiction.
We now go into the details, so we assume
$\widetilde{\omega }<\omega _0$.
Since $\Inv _{ \widetilde{\omega}   }$ is invariant by complex conjugation,
we deduce from~\eqref{fFinubu}  that
 $ \widetilde{\psi }^j,  \overline{ \widetilde{\psi }^j} \in \Inv _{ \widetilde{\omega}   }$, so that by~\eqref{fFlowd} and definition of $\omega _0$,
\begin{equation} \label{fEstlwru}
 \| \Sol (\cdot ) \widetilde{\psi }^j\|  _{ L^\Awr (\R, L^\Rwr) } <\infty ,
\end{equation}
for all $j\ge 1$.
Let
\begin{gather}
u_n(t)= \Sol (t) \phi _n, \label{fCritSu}  \\
v_n^j (t)=   \Sol (t-t_n^j)  \widetilde{\psi }^j (\cdot -x_n^j) , \label{fCritSubu} \\
u_n^\ell (t)=  \sum_{ j=1 }^\ell  v_n^j (t), \label{fCritSd} \\
\widetilde{W}_n^\ell = \sum_{j=1}^\ell  [\Esm{t_n^j}\psi^j(\cdot - x_n^j) -
\Sol(-t_n^j)\widetilde{\psi}^j(\cdot - x_n^j)] + W_n^\ell.  \label{fCritSt}
\end{gather}
It follows from~\eqref{fCritq} and~\eqref{fCritSu}--\eqref{fCritSt}  that
\begin{equation} \label{tilde}
i\partial_t u_n^\ell  + \Delta  u_n^\ell + |u_n^\ell |^{\Pmu } u_n^\ell  = e_n^\ell ,
\end{equation}
where
\begin{equation} \label{fCritSc}
e_n^\ell  =|u_n^\ell |^{\Pmu } u_n^\ell  -
\sum_{j=1}^\ell |v_n^j|^{\Pmu } v_n^j,
\end{equation}
and that
\begin{equation} \label{fCritSq}
u_n (0)-  u_n^\ell (0)= \widetilde{W}_n^\ell.
\end{equation}
We want to apply Proposition~\ref{LTPT}, and we begin by estimating $ \| u_n^\ell \| _{ L^\Awr ((0,\infty ), L^\Rwr) }$. It is not convenient to estimate this norm directly, so we estimate
$ \| u_n^\ell \| _{ L^\Gwr ((0,\infty ), L^\Gwr) }$ and $ \| u_n^\ell \| _{ L^\infty  ((0,\infty ), H^1) }$, the desired estimate resulting by using~\eqref{fHolu}.
We note that by~\eqref{fCritp}  and~\eqref{fCritn},
\begin{equation} \label{fCritSbu}
 \sum_{ j\ge 1 } \|\psi ^j\| _{ H^1 }^2 =: C_1< \infty .
\end{equation}
In particular, there exists $\ell _0\ge 1$ such that
\begin{equation} \label{fCritSbd}
 \|\psi ^j\| _{ H^1 }\le \frac {\Dsd } {2},\quad j\ge \ell _0,
\end{equation}
where $\Dsd $ is given by Proposition ~\ref{small}.
Applying~\eqref{fNLProfu}, we deduce from~\eqref{fCritSbd} that
for all $\ell \ge \ell _0$ there exists $n_1(\ell )\ge 1$ such that
\begin{equation} \label{fCritSbt}
 \|\Sol(-t_n^j)\widetilde{\psi}^j \| _{ H^1 } \le  \Dsd \le 1 ,\quad \ell _0 \le j\le \ell
 , n\ge n_1(\ell ).
\end{equation}
Furthermore, it follows from~\eqref{fCritSbu} and~\eqref{fNLProfu} that given any
$\ell \ge 1$ there exists $n_2(\ell) \ge 1$ such that
\begin{equation} \label{fCritSbq}
\sum_{ j= 1 }^\ell  \|\Sol(-t_n^j)\widetilde{\psi}^j \| _{ H^1 }^2 \le 2C_1,\quad n\ge n_2(\ell).
\end{equation}
It follows from~\eqref{fGerd} that, given any $\ell \ge \ell _0$,
\begin{equation} \label{fCritSs}
 \| \sum_{ \ell_0 }^\ell  v_n^j \| _{ L^\Gwr(\R^{\Dim +1}_+) }^\Gwr  \le
\sum_{ \ell_0 }^\ell   \|  v_n^j \| _{ L^\Gwr(\R^{\Dim +1}_+ ) }^\Gwr
+ C_\ell  \sum_{\ell_0\le  j\not = k \le \ell}   \int  _{\R^{\Dim +1}_+ }  |v_n^j|  |v_n^k|  |v_n^k|^{\Gwr -2}  ,
\end{equation}
with $\R^{\Dim +1}_+= (0,\infty )\times \R^\Dim $.
Applying~\eqref{fCritSubu}, \eqref{fsmallu}, \eqref{fCritSbt}, and~\eqref{fCritSbq}, we see that
\begin{equation} \label{fCritSp}
\begin{split}
\sum_{ \ell_0 }^\ell   \|  v_n^j \| _{ L^\Gwr(\R^{\Dim +1}_+ ) }^\Gwr  &
\le \Csd  \sum_{ \ell_0 }^\ell   \|\Sol(-t_n^j)\widetilde{\psi}^j \| _{ H^1 }^\Gwr
\\ & \le \Csd \sum_{ \ell_0 }^\ell   \|\Sol(-t_n^j)\widetilde{\psi}^j \| _{ H^1 }^2
\le 2 C_1 \Csd ,
\end{split}
\end{equation}
for $n\ge \max\{n_1(\ell ), n_2(\ell)  \}$.
Next, observe that, given $\ell _0 \le j\not = k\le \ell$,
\begin{equation*}
 \int  _{\R^{\Dim +1}_+ }  |v_n^j|  |v_n^k|  |v_n^k|^{\Gwr -2} \le
  \| v _n^k\| _{ L^\Gwr (\R^{\Dim +1}_+ ) }^{\Gwr -2}  \Bigl( \int  _{\R^{\Dim +1}_+ }  |v_n^j|^{\frac {\Gwr} {2}}  |v_n^k|^{\frac {\Gwr} {2}}  \Bigr)^{\frac {2} {\Gwr} }.
\end{equation*}
Applying~\eqref{fsmallu}, we obtain
\begin{multline*}
 \int  _{\R^{\Dim +1}_+ }  |v_n^j|  |v_n^k|  |v_n^k|^{\Gwr -2} \le
  \|\Sol(-t_n^k)  \widetilde{\psi }^k \| _{ H^1 }^{\Gwr -2}
  \\ \times
   \Bigl( \int  _{\R^{\Dim +1} }  | \Sol(t-t_n^j)\widetilde{\psi}^j(x - x_n^j)|^{\frac {\Gwr} {2}}  |\Sol(t -t_n^k)\widetilde{\psi}^k(x - x_n^j)|^{\frac {\Gwr} {2}}  \Bigr)^{\frac {2} {\Gwr} }.
\end{multline*}
Since $\Sol (t)  \widetilde{\psi }^j $ and $\Sol (t)  \widetilde{\psi }^k $ are two given functions of $L^\Gwr (\R^{\Dim +1} )$ and $ |t_n^k-t_n^j|+  |x_n^k-x_n^j|\to \infty $ as $n\to \infty $
by~\eqref{div}  if $ \widetilde{\psi }^j $ and $ \widetilde{\psi }^k $ are both nonzero, we
deduce from Lemma~\ref{eRCpu} that the right hand side of the above inequality converges to $0$ as $n\to \infty $; and so,
\begin{equation} \label{fCritSh}
 \int  _{\R^{\Dim +1}_+ }  |v_n^j|  |v_n^k|  |v_n^k|^{\Gwr -2} \goto  _{ n\to \infty  }0.
\end{equation}
We deduce from~\eqref{fCritSs}, \eqref{fCritSp} and~\eqref{fCritSh} that, given any $\ell \ge \ell _0$, there exists $n_3(\ell) \ge 1$ such that
\begin{equation} \label{fCritSn}
\| \sum_{ \ell_0 }^\ell  v_n^j \| _{ L^\Gwr(\R^{\Dim +1}_+) }^\Gwr  \le 4 C_1 \Csd  ,\quad
\ell \ge \ell_0, n\ge n_3(\ell).
\end{equation}
We now estimate the $H^1$ norm.
Note that
\begin{equation} \label{fCritSuz}
 \Bigl\| \sum_{  j=\ell_0 }^\ell v_n^j (t)  \Bigr\| _{ H^1 }^2=
 \sum_{  j=\ell_0 }^\ell  \| v_n^j (t)\|_{ H^1 }^2
 + 2\sum_{ \ell_0 \le j\not = k\le \ell }( v_n^j(t), v_n^k(t)) _{ H^1 }.
\end{equation}
Applying~\eqref{fCritSubu}, \eqref{fsmallu}, \eqref{fCritSbt}, and~\eqref{fCritSbq}, we see that
\begin{equation} \label{fCritSuu}
 \sum_{  j=\ell_0 }^\ell  \| v_n^j (t)\|_{ H^1 }^2  \le \Csd   \sum_{ \ell_0 }^\ell   \|\Sol(-t_n^j)\widetilde{\psi}^j \| _{ H^1 }^2
\le 2 C_1 \Csd  ,
\end{equation}
for $n\ge \max\{n_1(\ell ), n_2(\ell)  \}$.
Next, given any $j\not = k\ge \ell _0$ it follows from~\eqref{fCritSbt} and Corollary~\ref{eCritu}
(recall that $ |t_n^k-t_n^j|+  |x_n^k-x_n^j|\to \infty $ as $n\to \infty $
by~\eqref{div}  if $ \widetilde{\psi }^j $ and $ \widetilde{\psi }^k $ are both nonzero)
that
\begin{equation} \label{fCritSud}
\sup  _{ t\in \R } |( v_n^j(t), v_n^k(t)) _{ H^1 }| \goto _{ n\to \infty  }0.
\end{equation}
We deduce from~\eqref{fCritSuz}, \eqref{fCritSuu} and~\eqref{fCritSud} that, given any $\ell \ge \ell _0$, there exists $n_4(\ell) \ge 1$ such that
\begin{equation} \label{fCritSut}
\| \sum_{ \ell_0 }^\ell  v_n^j \| _{ L^\infty ((0,\infty ), H^1)}^2  \le 4 C_1 \Csd  ,\quad
\ell \ge \ell_0, n\ge n_4(\ell).
\end{equation}
Now if $n_5(\ell)= \max\{ n_3(\ell), n_4(\ell)\}$, it follows from~\eqref{fCritSn}, \eqref{fCritSut} and~\eqref{fHolu}  that there is a constant $A_1$ independent of $\ell \ge \ell _0$ such that
\begin{equation} \label{fCritSuq}
\| \sum_{ \ell_0 }^\ell  v_n^j \| _{ L^\Awr ((0,\infty ), L^\Rwr)} \le A_1, \quad \ell\ge \ell _0, n\ge n_5(\ell).
\end{equation}
On the other hand, applying~\eqref{fEstlwru} we see that there exists a constant $A_2$ such that
\begin{equation}  \label{fCritSuc}
\| \sum_{ j=1 }^{\ell _0} v_n^j \| _{ L^\Awr ((0,\infty ), L^\Rwr)} \le A_2,
\end{equation}
for all $n\ge 1$. Setting $A= A_1+A_2$, we deduce from~\eqref{fCritSuq}-\eqref{fCritSuc} that
\begin{equation} \label{fCritSus}
\| u_n^\ell \| _{ L^\Awr ((0,\infty ), L^\Rwr)} \le A, \quad \ n\ge n_5(\ell).
\end{equation}
We now fix $\ell _1$ sufficiently large so that
\begin{equation} \label{fCritSup}
\limsup _{ n\to \infty  } \| \Es{\cdot }  W_n^{\ell _1} \|_{L^{\Awr}((0,\infty ), L^{\Rwr })} \le \frac {\varepsilon (A)} {2},
\end{equation}
where $A$ is gven by~\eqref{fCritSus} and
 $\varepsilon (A)$ is given by Proposition~\ref{LTPT}. (Such an $\ell _1$ exists by~\eqref{fAsmall}.)
Applying~\eqref{fCritSt}, we obtain
\begin{multline*}
\|\Es{\cdot }  \widetilde{W}_n^{\ell _1}\|_{L^{\Awr}((0,\infty ), L^{\Rwr })}   \le
\sum_{j=1}^{\ell _1}\| \Esm{t_n^j}\psi^j  -
\Sol(-t_n^j) \widetilde{\psi}^j \|_{H^1} \\  +  \| \Es{\cdot }  W_n^{\ell _1} \|_{L^{\Awr}((0,\infty ), L^{\Rwr })},
\end{multline*}
so that, applying~\eqref{fNLProfu} and~\eqref{fCritSup}, there exists $n_6\ge 1$ such that
\begin{equation} \label{fCritSuh}
 \| \Es{\cdot }  \widetilde{W}_n^{\ell _1} \|_{L^{\Awr}((0,\infty ), L^{\Rwr })} \le  \varepsilon (A),\quad n\ge n_6.
\end{equation}
Finally, it is not difficult to show that for every $\ell \ge 2$ there exists a constant $C_\ell $ such that
for all $(z_j) _{ 1\le j\le \ell }\subset \mathbb{C}^\ell$
\begin{equation} \label{fGert}
 \Bigl| \,  \Bigl| \sum_{ j=1 }^\ell z_j  \Bigr|^{\Pmu } \sum_{ j=1 }^\ell z_j - \sum_{ j=1 }^\ell  |z_j|^{
 \Pmu } z_j\Bigr| \le C_\ell \sum_{ 1\le j\not = k\le \ell }  |z_j|^{\Pmu }  |z_k|.
\end{equation}
It follows from~\eqref{fGert} that
\begin{equation} \label{fFunald}
\| e_n ^{\ell _1}\|_{L^{\Bwr '}((0,\infty ), L^{\Rwr '})}  \le C _{ \ell _1 }
\sum_{ 1\le j\not = k\le \ell _1} \| \, |v_n^j|^{\Pmu }
|v_n^k| \,\|_{L^{\Bwr '}((0,\infty ), L^{\Rwr '})} .
\end{equation}
Fix $1\le j\not = k\le \ell _1$ and
note that by~\eqref{fEstlwru}
$S(\cdot ) \widetilde{\psi }^j, S(\cdot ) \widetilde{\psi }^k\in  L^\Awr (\R, L^\Rwr (\R^\Dim )) $.
Recall that $(\Pwr) \Bwr ' =\Awr$ and $(\Pwr) \Rwr '= \Rwr$.
Thus, approximating $S(\cdot ) \widetilde{\psi }^k$ in $L^\Awr (\R, L^\Rwr (\R^\Dim ))$ by functions of $C^\infty _\Comp (\R^{\Dim +1} )$ and using~\eqref{div} , it is easy to see that
\begin{equation} \label{fFunalt}
 \| \, |v_n^j|^{\Pmu } |v_n^k| \,\|_{L^{\Bwr '}((0,\infty ), L^{\Rwr '})}\goto  _{ n\to \infty  }0.
\end{equation}
Applying~\eqref{fFunald}-\eqref{fFunalt}, we deduce that there exists $n_7$ such that
\begin{equation} \label{fFunalq}
\| e_n ^{\ell _1}\|_{L^{\Bwr '}((0,\infty ), L^{\Rwr '})}  \le \varepsilon (A),\quad n\ge n_7.
\end{equation}
Using now~\eqref{fCritSus}, \eqref{fCritSuh} (together with~\eqref{fCritSq})
 and~\eqref{fFunalq}  we see that for every $n\ge \max\{
n_5(\ell _1), n_6, n_7 \}$, we have that
$\| u_n^{\ell _1}\| _{ L^\Awr ((0,\infty ), L^\Rwr)} \le A$,
$\| \Es{\cdot }  (u_n (0)-  u_n^{\ell _1}(0)) \|_{L^{\Awr}((0,\infty ), L^{\Rwr })} \le  \varepsilon (A)$ and $\| e_n ^{\ell _1}\|_{L^{\Bwr '}((0,\infty ), L^{\Rwr '})}  \le \varepsilon (A)$.
Applying Proposition~\ref{LTPT}, we conclude that  $\phi _n \in \Nonad$, which is absurd.

Thus we see that $ \widetilde{\omega } =\omega _0$.
It follows in particular from~\eqref{fDefoz} and~\eqref{fSupld} that $ \widetilde{M} =1$.
Applying~\eqref{fCritp}, we deduce that $\psi ^j=0$ for all $j\ge 2$, i.e. $J=2$ with the notation of Theorem~\ref{ProfileE}. Therefore, setting $\psi =\psi ^1$, $W_n= W_n^1$,
 $\tau _n=t_n^1$, $ \overline{\tau } = \overline{\tau }^1 $ and $x_n= x_n^1$ we see that  $M(\psi )=1$, \eqref{ePrinu:d} holds,  and~\eqref{ePrinu:u} follows from~\eqref{fCritq}.
Since $M(\phi _n)= M(\psi )=1$, it follows from~\eqref{fAsplitu} with $\lambda =0$ that
$ \|W_n \| _{ L^2 } \to 0$ as $n\to \infty $.
Next, since $ \widetilde{M} =1$ and $ \widetilde{\omega } =\omega _0$, we deduce from~\eqref{fDefoz} that $ \widetilde{E} = \omega _0 E(Q) M(Q)^\Siwr$;  and so by~\eqref{fSupld}
$\limsup  E( \Esm{\tau _n} \psi )= \omega _0 E(Q) M(Q)^\Siwr$ as $n\to \infty $.
Note that by~\eqref{eNLP:d} the lim sup is a limit, so that $E( \Esm{\tau _n} \psi )\to  \omega _0 E(Q) M(Q)^\Siwr$. Applying now~\eqref{fEsplit} and~\eqref{fYAM}  we obtain
$\limsup   E( W_n )=0$, and we deduce from~\eqref{fCritud} that $ \|\nabla W_n \| _{ L^2 } \to 0$ as $n\to \infty $. Therefore, $ \|W_n\| _{ H^1 } \to 0$. This completes the proof.
\end{proof}

\begin{proof}[Proof of Proposition~$\ref{critical}$]
We first show the existence of the critical solution.
By definition of $\omega _0$, there exist a sequence $(\omega _n) _{ n\ge 1 }\subset (0,1)$ and a sequence $(\phi _n) _{ n\ge 1 }$ such that 
$\phi _n\in \Inv  _{ \omega _n }$ and $\phi _n \not \in \Nonad$.
Note that the quantities $E(u) M(u)^\Siwr$ and $\| \nabla  u
\|_{L^2} \| u \|_{L^2}^ \Siwr $ are both invariant under the scaling
$u\mapsto \lambda ^{\frac {2} {\Pmu }} u(\lambda  \cdot )$.
Since~\eqref{SCH} is invariant under the scaling~\eqref{scaling}, we may assume
that $M(\phi _n)= 1$, and so we may apply Lemma~\ref{ePrinu}.
It follows from~\eqref{ePrinu:u} that
$ \phi_n =  \Esm{\tau _n} \psi (\cdot - x_n) + W_n$ for all $n \ge 1$.
Furthermore, we deduce from~\eqref{ePrinu:t} and~\eqref{fStrt} that
\begin{equation} \label{fCCpu}
 \| \Es{\cdot }  W_n \|_{L^\Awr ((0,\infty ), L^\Rwr )} \goto _{n \to \infty  }0.
\end{equation}
Let $ \widetilde{\psi } $ be the nonlinear profile associated to $\psi $ by Lemma~\ref{eNLP}.
Note that $M(\psi )=1$ so that by~\eqref{eNLP:t}, $M( \widetilde{\psi }) =1$.
Suppose
\begin{equation} \label{fCCpt}
 \| \Sol (\cdot )\widetilde{\psi }\| _{ L^\Awr (\R, L^\Rwr) }<\infty .
\end{equation}
We observe that
\begin{equation*}
\phi _n - \Sol(-\tau _n)  \widetilde{\psi } (\cdot -x_n)=
\Esm{\tau _n} \psi (\cdot - x_n ) - \Sol(-\tau _n)  \widetilde{\psi } (\cdot -x_n) + W_n,
\end{equation*}
so that applying~\eqref{eNLP:t}, \eqref{fStrt} and~\eqref{fCCpu}
\begin{equation} \label{fCCpq}
\| \Es{\cdot }  (\phi _n - \Sol(-\tau _n)  \widetilde{\psi } (\cdot -x_n) ) \|_{L^\Awr ((0,\infty ), L^\Rwr )}  \goto _{ n\to \infty  }0.
\end{equation}
It follows from~\eqref{fCCpt}, \eqref{fCCpq} and Proposition~\ref{LTPT} that $\phi _n\in \Nonad$  for all large $n$, which is absurd. (Recall that $\phi _n\not \in \Nonad$ by construction).
Thus~\eqref{fCCpt} fails, so that $\| \Sol (\cdot )\widetilde{\psi }\| _{ L^\Awr ((0,\infty ), L^\Rwr) }= \infty $ or $ \| \Sol (\cdot ) \overline{\widetilde{\psi }} \| _{ L^\Awr ((0,\infty ), L^\Rwr) }= \| \Sol (\cdot )\widetilde{\psi }\| _{ L^\Awr ((-\infty ,0), L^\Rwr) }= \infty $ (by~\eqref{fFlowd}).
In the first case we let $\Fcr =  \widetilde{\psi } $, and in the second case we let $\Fcr =   \overline{\widetilde{\psi }}  $. In both cases, the conclusions immediately follow.
(Note that $M(\Fcr)=1$ by Lemma~\ref{ePrinu}.)

We now prove the compactness property and we begin with the following claim.

\begin{claim} \label{fClaim}
For every sequence $(n_k) _{ k\ge 1 }\subset \N$, $n_k \to \infty $, there is a subsequence, which we still denote by $(n_k) _{ k\ge 1 }$, $(x_k) _{ k\ge 1 }\subset \R^\Dim $ and $v\in H^1 (\R^\Dim ) $ such that $\Ucr (n_k, \cdot -x_k)\to v$ in $H^1 (\R^\Dim ) $.
\end{claim}

To prove the claim, we apply Lemma~\ref{ePrinu}  to the sequence $\phi _k= \Ucr  (n_k)$
and we obtain
\begin{gather}
\Ucr (n_k) = \Esm{\tau _k} \psi (\cdot -x_k) + W_k,\quad k\ge 1\label{fComp:u} \\
\tau _k \goto _{ k\to \infty  } \overline{\tau }, \label{fComp:d}  \\
 \|W_k\| _{ H^1 } \goto _{ k\to \infty  }0. \label{fComp:t}
\end{gather}
We first assume $ \overline{\tau } =\infty $.
Note that $ \overline{\Ucr  (n_k)}  = \Es{\tau _k}  \overline{\psi}  (\cdot -x_k) +  \overline{W_k}$,
so that
\begin{equation}
 \| \Es{\cdot }  \overline{\Ucr  (n_k)}\|  _{ L^\Awr ((0,\infty ), L^\Rwr ) }
 \le  \| \Es{\cdot }  \overline{\psi } \|  _{ L^\Awr ((\tau _k,\infty ), L^\Rwr ) }
  +C  \|W_k\| _{ H^1 }.
\end{equation}
Since $\tau _k \to \infty $ and $ \|W_k\| _{ H^1 } \to 0$, we see that $ \| \Es{\cdot }  \overline{\Ucr  (n_k)}\|  _{ L^\Awr ((0,\infty ), L^\Rwr ) } \to 0$. In particular, for all sufficiently large $k$
\begin{equation}
\| \Es{\cdot }  \overline{\Ucr (n_k)}\|  _{ L^\Awr ((0,\infty ), L^\Rwr ) } \le \varepsilon (0),
\end{equation}
where $\varepsilon (\cdot )$ is given by Proposition~\ref{LTPT}.
Applying Proposition~\ref{LTPT} (with $ \widetilde{u} =e=0$ and $u= \Sol (\cdot ) \overline{\Ucr  (n_k)}$), we conclude that
\begin{equation*}
\| \Sol (\cdot ) \overline{\Ucr  (n_k)}\|  _{ L^\Awr ((0,\infty ), L^\Rwr ) } \le C (0),
\end{equation*}
for all large $k$. Note that by~\eqref{fFlowd}
$\Sol(t)  \overline{\Ucr (n_k)}=  \overline{ \Sol (-t+n_k)  \Fcr } $, so that the above inequality means
\begin{equation}
\| \Ucr  \|  _{ L^\Awr ((-\infty ,n_k), L^\Rwr ) } \le C (0).
\end{equation}
Letting $k\to \infty $ we obtain $\Ucr  \in L^\Awr (\R, L^\Rwr (\R^\Dim ))$. Thus
$\Fcr \in \Nonad$, which is absurd.
Therefore, we must have $ \overline{\tau } <\infty $ and we deduce from~\eqref{fComp:u}--\eqref{fComp:t} that  $ \Ucr (n_k, \cdot +x_k) \goto \Esm{ \overline{\tau } } \psi $, which proves our claim.

Next, we show that there exists a sequence $(y_n) _{ n\ge 1 }\subset \R^\Dim $ such that
the sequence $( \Ucr (n, \cdot -y_n)) _{ n\ge 1 }$ is relatively compact in $H^1 (\R^\Dim ) $.
We first observe there exist $R<\infty $ and $(y_n) _{ n\ge 1 }\subset \R^\Dim $ such that
\begin{equation} \label{fOu}
\int  _{  |x-y_n|<R } | \Ucr(n, \cdot -y_n)|^2 \ge \frac {3} {4} .
\end{equation}
Indeed, otherwise there exist a sequence $(n_k) _{ k\ge 1 }$ and a sequence $R_k\to \infty $ such that
\begin{equation} \label{fOd}
\sup  _{ y\in \R^\Dim  }\int  _{  |x-y|<R_k } | \Ucr (n_k, \cdot -y ) |^2 \le  \frac {3} {4} .
\end{equation}
If $n_k$ is bounded, this is absurd (since $\Ucr (n_k)$ belongs to a compact subset of $H^1 (\R^\Dim ) $), so we may assume $n_k\to \infty $.
By Claim~\ref{fClaim}, we deduce that, after possibly extracting, there exist $(x_k) _{ k\ge 1 }\subset \R^\Dim $ and $v\in H^1 (\R^\Dim ) $ such that $\Ucr (n_k, \cdot -x_k) \to v$. In particular, $ \|v\| _{ L^2 }=1$, so there exist $\rho >0$ and $z\in \R^\Dim $ such that
\begin{equation} \label{fOt}
\int  _{  |x-z|<\rho  } |v|^2 > \frac {3} {4} .
\end{equation}
It follows that
\begin{equation} \label{fOq}
\int  _{  |x-z|<\rho  } |\Ucr (n_k, \cdot -x_k) |^2 > \frac {3} {4} ,
\end{equation}
for $k$ large. This contradicts~\eqref{fOd}, thus proving~\eqref{fOu}.
We now show that the sequence $(y_n) _{ n\ge 1 }$ has the desired property.
Indeed,  consider a sequence
$n_k\to \infty $. By Claim~\ref{fClaim}, we see that (after possibly extracting) there exist $(x_k) _{ k\ge 1 }\subset \R^\Dim $ and $v\in H^1 (\R^\Dim ) $ such that $\Ucr (n_k, \cdot -x_k) \to v$.
Arguing as above, we deduce that there exists $\rho >0$ such that~\eqref{fOt} holds.
Since $M(\Fcr)=1$, \eqref{fOu} and~\eqref{fOt} show that $ |y _{ n_k }- x_k|< R+\rho $.
Thus by possibly extracting we may assume that $y _{ n_k }- x_k \to  \overline{x} \in \R^N $ as $k\to \infty $. It easily follows that $\Ucr (n_k, \cdot -y _{ n_k }) \to v(\cdot - \overline{x} )$ in $H^1 (\R^\Dim ) $ as $k\to \infty $. Thus we have proved that any subsequence of $( \Ucr (n, \cdot -y_n)) _{ n\ge 1 }$ has a converging subsequence; and so
  $( \Ucr (n, \cdot -y_n)) _{ n\ge 1 }$ is relatively compact in $H^1 (\R^\Dim ) $.

It now follows from  Remark~\ref{eRemu} that $(\Sol(1) \Ucr  (n, \cdot -y_n)) _{ n\ge 1 }$ is also relatively compact in $H^1 (\R^\Dim ) $.
Since $\Sol (1)\Ucr  (n, \cdot -y_n)= \Ucr  (n+1, \cdot -y_n)$, we see that both the sequences $( \Ucr  (n, \cdot -y_n)) _{ n\ge 1 }$ and $( \Ucr (n, \cdot -y_{n-1})) _{ n\ge 2 }$ are relatively compact.
In other words, if we set $v_n= \Ucr (n, \cdot -y_n)$, then $(v_n) _{ n\ge 1 }$ and $(v_n (\cdot -y_{n-1} +y_n)) _{ n\ge 2 }$ are both relatively compact.
Since $ \|v_n\| _{ L^2 } =  \| \Fcr \| _{ L^2 }=1$, it easily follows that
\begin{equation*}
R: = \sup _{ n\ge 2 } |y_n- y _{ n-1 }| <\infty .
\end{equation*}
Note that by Remark~\ref{eRemu},
\begin{equation*}
E= \{ \Ucr  (n+\theta , \cdot -y_n+y);\, n\ge 1, 0\le \theta \le 1,  |y|\le R \},
\end{equation*}
is relatively compact. We finally define $x\in C([0,\infty ))$ for $t\ge 0$ by $x(t)= y_1$ for $0\le t<1$ and $x(t)= y_n+ (t-n) (y _{ n+1 }-y_n)$ for $n\le t<n+1$, $n\ge 1$.
Since $\Ucr (t, \cdot -x(t)) \in E$ for all $t\ge 1$, we see that $\union  _{ t\ge 0 } \{ \Ucr  (t, \cdot -x(t)) \}$ is relatively compact, and this completes the proof.
\end{proof}

\section{Rigidity} \label{sRigidity}

The main result of this section is the following rigidity, or Liouville-type theorem, which implies in particular that the critical solution constructed in Proposition~\ref{critical} must be identically zero.
 It is similar to  Theorem~6.1 in~\cite{DHR} which concerns the 3D cubic NLS.
 The proof of~\cite{DHR} is easily adapted to the present situation, and we give it for completeness.

\begin{thm}\label{Rigidity}
Let $\DI \in \Inv$ and $u\in C([0,\infty ), H^1(\R^\Dim ))$ the corresponding solution of~\eqref{SCH}.  If there exists a function $x\in C([0,\infty ),\R^\Dim )$ such that
\begin{equation} \label{ep}
\sup  _{ t\ge 0 } \int  _{ \{  |x+x(t)|>R \} }\{  |\nabla u(t,x)|^2 +  |u(t,x)|^{\Ppu } + |u(t,x)|^2 \}
\goto  _{ R\to \infty  }0,
\end{equation}
then $\DI  = 0$.
\end{thm}

We will use the following local version of the Virial identity.

\begin{lem}\label{local virial}
Let $\chi \in C^\infty _\Comp ([0,\infty ))$ and $T>0$. If $u\in C([0,T], H^1
(\R^\Dim  ) )$ is a solution of~\eqref{SCH}, then the function $t\mapsto \int _{ \R^\Dim  }
\chi( r ) |u(t,x)|^2 dx $ belongs to $C^2([0,T])$. Moreover,
\begin{equation*}
\frac {d} {dt} \int _{ \R^\Dim } \chi( r) |u(t,x)|^2 dx= 2 \im \int  _{ \R^\Dim  }
\chi '( r )  \overline{u} \frac {\partial u} {\partial r}\, dx,
\end{equation*}
and
\begin{multline*}
\frac {d ^2} {dt^2} \int_{ \R^\Dim } \chi(r) |u(t,x)|^2 dx =
8\| \nabla  u \|_{L^2}^2 - \frac{4 \Dim \Pmu}{\Ppu }\| u \|_{L^{\Ppu }}^{\Ppu }
\\
+ 4 \int _{ \R^\Dim } \Bigl( \frac {\chi '(r)} {r} -2 \Bigr)  |\nabla u|^2
+ 4 \int  _{ \R^\Dim } \Bigl( \chi '' (r) - \frac {\chi '(r)} {r} \Bigr)  \Bigl| \frac {\partial u} {\partial r} \Bigr| ^2
\\ +   \frac {2 \Pmu } {\Ppu } \int _{ \R^\Dim } (2\Dim -\Delta \chi )  |u|^{\Ppu }
-  \int _{ \R^\Dim }  |u|^2 \Delta ^2\chi  ,
\end{multline*}
for all $0\le t\le T$. (In the above formulas, $\chi $ is considered either as a function of $r= |x|$ or as a function of $x$.)
\end{lem}

\begin{proof}
If $\DI \in H^2( \R^\Dim )$, then $u\in C([0,T], H^2( \R^\Dim )) \cap C^1([0,T], L^2( \R^\Dim ))$
and the result follows from direct calculations. See Lemma~2.9 in~\cite{Kavian}
or formula~(6.5.35) in~\cite{Css}. The general case follows by approximating $\DI $ by smooth functions and from continuous dependence.
\end{proof}

\begin{proof}[Proof of Theorem~$\ref{Rigidity}$]
We assume by contradiction that $\DI \not = 0$ and
we first note that we may assume without loss of generality that $\DI $ has null momentum, i.e.
\begin{equation} \label{fMomu}
P(\DI )= 0.
\end{equation}
Indeed, set
\begin{equation} \label{fGalu}
y_0= - \frac {P(\DI )} {M(\DI )},
\end{equation}
and let $ \widetilde{\DI} $ be defined by
\begin{equation}  \label{fGald}
  \widetilde{\DI}  (x)= e^{i x \cdot  y_0} \DI (x).
\end{equation}
By elementary calculations,  $  \widetilde{\DI} \in H^1 (\R^\Dim )  $,
$ \|  \widetilde{\DI}  \| _{ L^2 }=  \|\DI \| _{ L^2 }$, $ \| \widetilde{\DI}  \| _{ L^{\Ppu } }=  \|\DI \| _{ L^{\Ppu } }$ and $ \|\nabla   \widetilde{\DI} \| _{ L^2 }^2 =  \|\nabla \DI \| _{ L^2 }^2 - 2 |P(\DI )|^2 M(\DI )^{-1}\le  \|\nabla \DI \| _{ L^2 }^2$. In particular, we see that $ \widetilde{\DI} \not = 0$ and $  \widetilde{\DI}  \in \Inv$. Moreover, by Galilean invariance, the corresponding solution
$ \widetilde{u} \in C([0,\infty ), H^1(\R^\Dim ))$ of~\eqref{SCH} is given by
\begin{equation*}
 \widetilde{u} (t,x)= e^{ i(x\cdot y_0 -t  |y_0|^2) } u(t,x -2ty_0).
\end{equation*}
It follows easily that there exists a constant $C$ independent of $t,x$ such that
\begin{equation*}
[ |\nabla  \widetilde{u} |^2 +  | \widetilde{u} |^{\Ppu } +  | \widetilde{u} |^2 ] (t,x)
\le C [ |\nabla u|^2 +  |u|^{\Ppu } + |u|^2] (t, x-2ty_0),
\end{equation*}
so that $ \widetilde{u} $ satisfies the assumption~\eqref{ep} with $x(t)$ replaced by $ \widetilde{x} (t)= x(t)- 2t y_0$.
Thus we see that $  \widetilde{\DI} $ satisfies all the assumptions of the theorem, along with the null momentum condition~\eqref{fMomu}.

We now proceed in two steps.

\medskip
\Step1 We show that
\begin{equation} \label{fLocal}
\frac { |x(t)|} {t} \goto  _{ t\to \infty  }0.
\end{equation}
Indeed, otherwise there exist $\delta  > 0$ and a sequence $t_n
\to \infty$  such that
\begin{equation}
|x(t_n)| \ge \delta  t_n.
\end{equation}
Without loss of generality, we may suppose $x(0) = 0$.
Let
\begin{equation} \label{fDftanu}
\tau _n = \inf \left\{ t \ge 0; |x(t)| \ge |x(t_n)| \right\}.
\end{equation}
It follows easily from~\eqref{fDftanu} that $0<\tau _n\le t_n$ and $ |x(\tau _n)|=  |x(t_n)|$, so that
\begin{align}
& \tau _n \goto _{ n\to \infty  }\infty , \label{fDftandb} \\
& |x(t)| <  |x(\tau _n)|,\quad 0\le t<\tau _n,
\label{fDftand} \\
& |x(\tau _n)| \ge   \delta \tau _n. \label{fDftant}
\end{align}
Fix a function $\theta \in C^\infty ([0,\infty ))$ such that $0\le \theta \le 1$, $ |\theta '|\le 2$ and
\begin{equation*}
\theta (r)=
\begin{cases}
1 & 0\le r\le 1\\
0 & r\ge 2.
\end{cases}
\end{equation*}
Given $R>0$, set
\begin{equation*}
\theta _R(r)= \theta  \Bigl( \frac {r} {R} \Bigr).
\end{equation*}
One verifies easily that
\begin{equation} \label{fTdeter}
 |\theta _R( |x|)-1| +  |x|\,  |\theta _R ' ( |x|)| \le 5 \times 1 _{ \{  |x|>R \} }
 \quad  \text{and}\quad  |x| \theta _R( |x|) \le 2R.
\end{equation}
Let
\begin{equation*}
z_R(t) = \int  _{ \R^\Dim  } x \theta _R(  |x|) |u(t,x)|^2 dx.
\end{equation*}
Multiplying the equation~\eqref{SCH} by $x\theta _R  \overline{u} $, we obtain by  an easy calculation that
\begin{equation*}
z_R '(t)= 2\im \int  _{ \R^\Dim  }  \Bigl\{ \theta _R   \overline{u} \nabla u  + x \theta _R'  \overline{u} \frac {\partial u} {\partial r} \Bigr\} = 2\im \int  _{ \R^\Dim  }  \Bigl\{ (\theta _R -1)  \overline{u} \nabla u  +x \theta _R'  \overline{u} \frac {\partial u} {\partial r} \Bigr\},
\end{equation*}
where we used the property $P(u(t))=0$ (see~\eqref{fMomu}) in the last identity.
Applying~\eqref{fTdeter}, we deduce that
\begin{equation}  \label{fTdeterbu}
|z_R'(t)| \le 10 \int  _{ \{  |x|>R \} }  |u|\,  |\nabla u|\le 5\int  _{ \{  |x|>R \} } \{ |\nabla u|^2 +  |u|^2 \}.
\end{equation}
On the other hand, it follows from~\eqref{ep} that
 there exists $\rho >0$ such that
\begin{equation} \label{fTdeterbd}
\int_{ \{ |x + x(t)|> \rho  \} } \{  | \nabla  u(t) |^2 +  |u(t)|^2 \} \le
\frac{ \delta  M(\DI )}{10 (1+\delta )},
\end{equation}
for all $t\ge 0$. Set
\begin{equation}  \label{fTdeterbt}
R_n =  |x(\tau _n)| + \rho .
\end{equation}
Given $0\le t\le \tau _n$ and $ |x|>R_n$, we deduce from~\eqref{fDftand} and~\eqref{fTdeterbt} that
\begin{equation} \label{fTdeterbq}
|x + x(t)|\ge R_n -  |x(t)|\ge R_n-  |x(\tau _n)| =\rho .
\end{equation}
Applying~\eqref{fTdeterbu}, \eqref{fTdeterbq} and~\eqref{fTdeterbd}, we obtain
\begin{equation}\label{pt}
\begin{split}
|z_{R_n} '(t)| &\le \frac{ \delta  M(\DI )}{2 (1+\delta )},
\end{split}
\end{equation}
for all $n\ge 1$ and $0\le t\le \tau _n$.
Next, since $R_n\ge \rho $ and $x(0)=0$,
\begin{equation}\label{fz}
\begin{split}
|z_{R_n} (0)|&\le \int_{ \{ |x|<\rho \} } |x| \theta  _{ R_n } |\DI  |^2 +
\int_{ \{ |x | > \rho \} }  |x| \theta  _{ R_n } |\DI |^2 \\
&= \int_{ \{ |x|<\rho \} } |x| \,  |\DI  |^2  +
\int_{ \{ |x +x(0)| > \rho \} }  |x| \theta  _{ R_n } |\DI |^2 \\
&\le \rho M(\DI ) +  \frac{ \delta  M(\DI )}{5 (1+\delta )} R_n,
\end{split}
\end{equation}
where we used~\eqref{fTdeter} and~\eqref{fTdeterbd} the  in the last estimate.
We now estimate $z _{ R_n }(\tau _n)$ as follows.
\begin{equation} \label{fTdeterbc}
\begin{split}
z _{ R_n }(\tau _n) &= \int_{ \{ |x + x(\tau _n)|> \rho \} }  x \theta  _{ R_n } |u(\tau _n, x)|^2
+ \int_{ \{ |x + x(\tau _n)|< \rho \} }  x \theta  _{ R_n } |u(\tau _n, x)|^2 \\ & =
\mathrm{I} + \mathrm{II}.
\end{split}
\end{equation}
Applying~\eqref{fTdeter} and~\eqref{fTdeterbd}, we see that
\begin{equation} \label{fTdeterbs}
|\mathrm{I}| \le  \frac{ \delta  M(\DI )}{5 (1+\delta )} R_n.
\end{equation}
Next, we note that if $ |x + x(\tau _n)| < \rho $, then
 $|x|\le |x + x(\tau _n)| + |x(\tau _n)|\le \rho + |x(\tau _n)|= R_n$,
 so that $\theta  _{ R_n } ( |x|)=1$; and so
\begin{multline*}
\mathrm{II}
=  \int_{ \{ |x + x(\tau _n)|< \rho \} }  x  |u(\tau _n, x)|^2 \\
= \int_{ \{ |x + x(\tau _n)| < \rho \}} (x +
x( \tau _n )) |u(\tau _n, x)|^2  - x(\tau _n)\int_{ \{ |x + x(\tau _n)| <
\rho \} }|u(\tau _n,x)|^2 \\
= -  x(\tau _n) M(\DI ) + \int_{ \{ |x + x(\tau _n)| < \rho \}} (x +
x( \tau _n )) |u(\tau _n, x)|^2
 \\ + x(\tau _n)\int_{ \{ |x + x(\tau _n)| >
\rho \} }|u(\tau _n,x)|^2 .
\end{multline*}
Applying~\eqref{fTdeterbt} and~\eqref{fTdeterbd}, we deduce that
\begin{equation} \label{fTdeterbp}
|\mathrm{II}|\ge |x(\tau _n)|M(\DI ) - \rho M(\DI )- \frac{ \delta  M(\DI )}{10 (1+\delta )} R_n.
\end{equation}
Estimates~\eqref{fTdeterbc}, \eqref{fTdeterbs} and~\eqref{fTdeterbp} yield
\begin{equation} \label{ft}
|z_{R_n} (\tau _n)| \ge  |x(\tau _n)|M(\DI ) - \rho M(\DI )- \frac{ 3 \delta  M(\DI )}{10 (1+\delta )} R_n.
\end{equation}
We deduce from~\eqref{pt}, \eqref{ft} and~\eqref{fz}  that
\begin{equation*}
\begin{split}
\frac{ \delta  M(\DI )}{2 (1+\delta )} \tau _n & \ge \int _0^{\tau _n}  |z _{ R_n }' (t)|\, dt
\ge  |z _{ R_n } (\tau _n)| -   |z _{ R_n } (0)| \\ &
\ge  |x(\tau _n)|M(\DI ) - 2 \rho M(\DI )- \frac{ \delta  M(\DI )}{2 (1+\delta )} R_n
\end{split}
\end{equation*}
Finally, applying~\eqref{fTdeterbt}, we obtain
\begin{equation*}
\frac { |x(\tau _n)|} {\tau _n} \le \frac {\delta } {2+\delta }+ \frac {\rho } {\tau _n} \frac {4+5\delta } {2+\delta },
\end{equation*}
which, together with~\eqref{fDftant} and~\eqref{fDftandb},  yields a contradiction.

\medskip
\Step2 Conclusion. Since $\DI \not = 0$, we have
\begin{equation} \label{fEnerp}
E(\DI )>0,
\end{equation}
by Lemma~\ref{Cdelta}.
Fix a function $\chi \in C^\infty ([0,\infty ))$ such that
\begin{equation*}
\chi   (r)=
\begin{cases}
r^2,& \textrm{if } r \le 1, \\
0,& \textrm{if } r\ge2.
\end{cases}
\end{equation*}
Given $R \ge 1$, set
\begin{equation*}
\chi _R(r)= R^2 \chi    \Bigl( \frac{ r} {R} \Bigr) ,
\end{equation*}
so that
\begin{equation*}
\frac {\chi _R' (r)} {r}-2= \chi ''(r)- \frac {\chi _R' (r)} {r}= 2\Dim- \Delta \chi _R ( |x|) = \Delta ^2 \chi _R
( |x|)=0 \quad  \text{for }r\le R,
\end{equation*}
and
\begin{equation*}
\frac {1} {R}   \|\chi _R'\| _{ L^\infty  }+
 \Bigl\| \frac {\chi _R' } {r}\Bigr\| _{ L^\infty  } +  \|\chi _R''\| _{ L^\infty  }+  \|\Delta \chi _R
 ( |\cdot |)\| _{ L^\infty  }
+ R^2  \|\Delta ^2\chi _R ( |\cdot |)\| _{ L^\infty  } \le C,
\end{equation*}
for some constant $C$ independent of $R>0$.
Define
\begin{equation*}
Z_R(t) = \int  _{ \R^\Dim } \chi _R(  |x| ) |u(t,x)|^2dx.
\end{equation*}
It follows from Lemma~\ref{local virial} that $Z_R \in C^2([0,\infty ))$ and that
\begin{equation}\label{zp}
|Z_R'(t)|  \le C R \|u(t)\| _{ L^2  }  \|\nabla u(t)\| _{ L^2 }
 \le A R,
\end{equation}
for some constant $A$ independent of $t\ge 0$ and $R>0$.
Moreover,
\begin{equation} \label{fLiou}
Z_R ''(t) = 8\| \nabla  u(t) \|_{L^2}^2 - \frac{4 \Dim \Pmu}{\Ppu }\| u (t) \|_{L^{\Ppu }}^{\Ppu }
+ H_R (u(t)),
\end{equation}
where
\begin{equation} \label{fLiod}
|H_R (u(t))|\le B\int_{ \{ |x|\ge R \} } \{  | \nabla  u(t) |^2 + |u(t)|^{\Ppu }+ |u(t)|^2\} ,
\end{equation}
for some constant $B$ independent of $t\ge 0$ and $R>0$.
Set
\begin{equation*}
\eta = 8  \Bigl[ 1 -   \Bigl( \frac {E(\DI )M(\DI )^\Siwr } {E(Q)M(Q)^\Siwr } \Bigr)^{ \frac {\Dim \Pmu -4} {4} }  \Bigr] >0.
\end{equation*}
We deduce from Lemma~\ref{Cdelta} that
\begin{equation} \label{fLiot}
8 \| \nabla  u (t)\|_{L^2}^2 -\frac{4\Dim \Pmu}{\Ppu }\| u (t) \|_{L^{\Ppu }}^{\Ppu } \ge
\eta \|\nabla  u (t)\|_{L^2}^2 \ge 2 \eta E(u(t)) =  2 \eta E(\DI ).
\end{equation}
Since $E(\DI ) >0$
by~\eqref{fEnerp}, it follows from~\eqref{ep} that
 there exists $\rho \ge  1$ such that
\begin{equation} \label{fLioq}
\int_{ \{ |x + x(t)|\ge \rho \} } \{  | \nabla  u(t) |^2 + |u(t)|^{\Ppu }+ |u(t)|^2 \} \le
\frac{\eta E(\DI )}{B},
\end{equation}
for all $t\ge 0$, where $B$ is the constant in~\eqref{fLiod}.
Next, we deduce from~\eqref{fLocal}  that there exists $t_0>0$ such
that
\begin{equation} \label{fLioh}
|x(t)| \le \frac{\eta E(\DI )}{4A}   t,
\end{equation}
for $t\ge t_0$. Given $\tau >t_0$, set
\begin{equation} \label{fLios}
R_\tau  = \rho  + \frac{\eta E(\DI )}{4A}  \tau
\end{equation}
It follows easily from~\eqref{fLioh} and~\eqref{fLios}
that $\{|x|\ge R _\tau \}\subset\{
|x+x(t)|\ge \rho  \}$ for $t\in [t_0, \tau ]$, thus by~\eqref{fLiod} and~\eqref{fLioq}
\begin{equation} \label{fLioc}
|H  _{ R_\tau  }(u (t))|\le \eta E(\DI ),
\end{equation}
for all $t\in [t_0, \tau ]$. We deduce from~\eqref{fLiou}, \eqref{fLiot} and~\eqref{fLioc}
\begin{equation}\label{2p}
Z_{R_\tau }''(t) \ge \eta E(\DI ),
\end{equation}
for all $t\in [t_0, \tau ]$. Integrating \eqref{2p} on $(t_0,\tau )$
and applying~\eqref{zp} and~\eqref{fLios}   yields
\begin{equation}\label{le}
\begin{split}
\eta E(\DI )(\tau  - t_0) & \le \int_{t_0}^{\tau }Z_R''(t)dt \le
|Z_{R_\tau } '(\tau ) - Z_{R_\tau }'(t_0)| \\
&\le 2AR_\tau =2A  \Bigl( \rho  + \frac{\eta E(\DI )}{4A}  \tau  \Bigr).
\end{split}
\end{equation}
Letting $\tau \to \infty $ in~\eqref{le} yields a contradiction.
\end{proof}

\section{Appendix: a Gronwall-type inequality} \label{sGron}

\begin{lem}  \label{gronwall}
Let  $1\le \beta <\gamma \le \infty $,  $0<T\le \infty $ and let  $f\in L^\rho
(0,T )$, where  $1\le \rho <\infty $ is defined by  $\frac {1} {\rho }=
\frac {1} {\beta } - \frac {1} {\gamma }$. If  $\eta \ge 0$ and $\varphi \in
L^\gamma  _{\mathrm {loc} }([0,T )) $ satisfy
\begin{equation}  \label{fHyp}
 \|\varphi \| _{ L^\gamma (0,t) } \le \eta +  \|f\varphi \| _{ L^\beta (0,t) },
\end{equation}
for all  $0<t<T$, then
\begin{equation}  \label{fConcl}
 \|\varphi \|_{ L^\gamma (0,t) } \le \eta \Phi ( \| f \| _{ L^\rho (0,t) })
\end{equation}
for all  $0<t\le T$, where $\Phi (s)\equiv  2 \Gamma (3+2s)$ and
$\Gamma $ is the Gamma function.
\end{lem}

\begin{proof}
Since  $f\in L^\rho (0,T)$, there exist  $\ell \le 2  \|f\| _{ L^\rho (0,T)
}$ and an increasing sequence  $(\tau _k) _{ 0\le k\le \ell }$ such
that  $\tau _0=0$,  $\tau _\ell= T$ and
\begin{equation}  \label{fDeftk}
 \|f\| _{ L^\rho (\tau  _{ k-1 },\tau _k) }=\frac {1} {2}  \text{ if } \ell \ge 2 \text{ and }   1\le k\le \ell-1 ,
 \quad  \|f\| _{ L^\rho (\tau  _{ \ell-1 },\tau _\ell) }\le \frac {1} {2}.
\end{equation}
Set  $a_0=0$ and  $a_k=  \|\varphi  \| _{ L^\gamma (0, \tau _k )}$ for
$1\le k\le \ell$. It follows from~\eqref{fHyp}, H\"older's
inequality and~\eqref{fDeftk}  that for all  $k\le \ell-1$
\begin{equation*}
\begin{split}
a  _{ k+1 } &\le \eta +  \|f\varphi \| _{ L^\beta (0,\tau _k) }+ \|f\varphi \| _{ L^\beta (\tau _k,\tau  _{ k+1 }) } \\
& \le \eta +  \|f \| _{ L^\rho (0,\tau _k) }  \|\varphi \| _{ L^\gamma (0,\tau
_k) } +
\|f \| _{ L^\rho (\tau _k,\tau  _{ k+1 }) } \|\varphi \| _{ L^\gamma (\tau _k,\tau  _{ k+1 }) } \\
& \le \eta + \frac {k} {2} a_k + \frac {1} {2} a _{ k+1 }.
\end{split}
\end{equation*}
Thus we see that  $a _{ k+1 } \le 2 \eta + ka_k$, so that  $a_k \le
2\eta (k+1)!$. Let  $0<t<T$ and   $1\le k\le \ell$ be such that
$\tau  _{ k-1 }\le t<\tau  _k$. It follows that
\begin{equation}   \label{fDeftkb}
 \|\varphi \| _{ L^\gamma  (0,t) }\le a_k \le 2\eta (k+1)!
\end{equation}
On the other hand, we deduce from~\eqref{fDeftk} that
\begin{equation*}
 \|f\| _{ L^\rho (0,t) } \ge  \|f\| _{ L^r\rho (0,\tau  _k)}
 \ge \frac {k} {2}-\frac {1} {2},
\end{equation*}
thus  $(k+1)! = \Gamma (k+2) \le \Gamma (3+ 2 \|f\| _{ L^\rho (0,t) })$.
The result now follows from~\eqref{fDeftkb}.
\end{proof}

\end{document}